%% file: STDmap-arxiv.tex
\documentclass{amsart}
\usepackage{amssymb,latexsym,amsthm, amsmath}
\usepackage{algorithmic,algorithm}
\usepackage{graphicx}
\usepackage{tikz}
\usepackage{enumerate}

\ifx\pdfoutput\undefined
\usepackage{hyperref}
\else
\usepackage[pdftex,colorlinks=true,linkcolor=blue,urlcolor=blue]{hyperref}
\fi
\theoremstyle{plain}
\newtheorem{theorem}{Theorem}[section]
\newtheorem{corollary}[theorem]{Corollary}

\theoremstyle{definition}
\newtheorem*{ack}{Acknowledgements}
\newtheorem{definition}[theorem]{Definition}
\newcommand{\rafnote}[1]{\textsf{\color{blue}\textbf{Raf Note:} #1}}
\newcommand{\rodnote}[1]{\textsf{\color{red}\textbf{Rodrigo Note:} #1}}
\renewcommand{\rafnote}[1]{} % uncomment me to make the notes disappear
\renewcommand{\rodnote}[1]{} % uncomment me to make the notes disappear
\begin{document}
\title[Efficient Automation of Index Pairs]
      {Efficient Automation of Index Pairs in Computational Conley Index Theory}
\author{Rafael Frongillo}
\address{Computer Science Department\\
         University of California\\
         Berkeley, CA 94720}
\email{raf@cs.berkeley.edu}
\author{Rodrigo Trevi\~no}
\address{Department of Mathematics\\
         The University of Maryland\\
         College Park, MD 20742}
\email{rodrigo@math.umd.edu}
\thanks{R.T. is supported by the Brin and Flagship Fellowships at the University of Maryland.}
\date{\today}
\begin{abstract}
  We present new methods of automating the construction of index
  pairs, essential ingredients of discrete Conley index theory.  These
  new algorithms are further steps in the direction of automating
  computer-assisted proofs of semi-conjugacies from a map on a
  manifold to a subshift of finite type. We apply these new algorithms
  to the standard map at different values of the perturbative
  parameter $\varepsilon$ and obtain rigorous lower bounds for its
  topological entropy for $\varepsilon\in[.7,2]$.
\end{abstract}

\maketitle

\section{Introduction}

The literature on computer-assisted proofs in dynamical systems through topological methods involves, for the most part, the investigation of the mapping properties of certain compact sets which go under different names in different settings: they are called \emph{index pairs} in Conley index theory and \emph{windows} or \emph{h-sets} in the theory of correctly aligned windows or covering relations \cite{covering1}, which are the two leading theories applied in topological, computer-assisted proofs in dynamical systems. An index pair is not necessarily an h-set and vice-versa, but they both satisfy the crucial property that they map across themselves in a way similar to the way Smale's horseshoe maps across itself. In other words, they agree in some sense with the expanding and contracting directions of the map and thus they can be thought of as slight generalizations of Markov partitions. Not surprisingly, after several conditions are met in each setting, one is able to prove a semi-conjugacy to symbolic dynamics, giving a good description the dynamics of a large set of the orbits of a dynamical system, much like one can give a good description of the dynamics of Smale's horseshoe via the easily-describable dynamics of the full 2-shift. Given an index pair, an automated procedure was described in~\cite{DFT} to prove a semi-conjugacy to a sub-shift of finite type (SFT).

The first major push towards automated validation and automated computer-assisted proofs using topological methods can be found in~\cite{DJM05}.  By \emph{validation} we mean a mechanism which upholds the claim of a theorem in rigorous way by means of finitely many calculations on a computer. Built on those and other results, \cite{DFT} achieved the complete automation of the validation part of a computer-assisted proof. In this paper we address the issue of automating the creation of index pairs in a computationally efficient way.

There are many reasons why automated proof techniques such as~\cite{DFT} come at an advantage. For example, for parameter-depending systems, automation allows for the rigorous exploration of a system at many parameters. Such applications of~\cite{DFT} have already been carried out in~\cite{raf-plats}. Besides exploration at different values of parameters, automated methods may permit us to construct index pairs where it would be otherwise impossible to do so by hand. In~\cite{DFT} computer-assisted proofs were given of a semi-conjugacy from the classical H\'{e}non map to a SFT with 247 symbols. These proofs were completely automated, that is, besides an input of initial parameters, there was no human intervention from start to finish in the computation which proved the semi-conjugacy and bounded the topological entropy.

It should be noted that the complete automation described in the previous paragraph is remarkable. To prove a semi-conjugacy using topological methods on a computer the following steps must be taken: 
\begin{enumerate}
\item Locate and identify the relevant invariant objects; 
\item construct a compact set $K$ (refered to above, in our case the index pair) whose components map onto each other in a way which is compatible with the dynamics of the system; 
\item prove that some of the dynamics of the system can be described via the mapping properties between the components of $K$. 
\end{enumerate}
Each of these steps presents its own challenges and difficulties. In \cite[\S 3.2]{DFT} step (3) was completely solved, in the sense that given any compact set (index pair) $K$ which satisfies the properties required by (2), a completely automated routine was given to prove a semi-conjugacy. Steps (1) and (2), the creation of the index pair, were also fully automated in \cite[\S 3.1]{DFT} (and~\cite{raf-plats}), largely due to the properties of ``well-behaved'' maps like the H\'{e}non map: it has a low-dimensional attractor and all its periodic orbits exhibit hyperbolic-like behavior from a computational perspective. In other words, in such cases there is a dominating invariant set, namely the attractor whose periodic orbits all seem hyperbolic and are easily localizable. Thus, it is of interest to consider the case where we do not have such underlying structure. Quoting directly from \cite[\S 5]{DFT}: ``\emph{... further analysis and optimization of the procedure described in section 3.1 for locating a region of interest should lead to even stronger results.}'' 

This article is an effort to extend the methods and techniques presented in~\cite{DFT} to more general situations with the aim of dealing with systems which do not exhibit H\'{e}non-like behavior.  As an illustration of the approach we propose, we apply our method to the standard map 
\[f_\varepsilon:(x,y)\mapsto\left(x+y+\frac{\varepsilon}{2\pi}\sin(2\pi x)\mbox{ mod 1 }, y+\frac{\varepsilon}{2\pi}\sin(2\pi x)\mbox{ mod 1 }\right),\]
for $(x,y) \in S^1\times S^1 = \mathbb{T}^1$, at different values of the perturbation parameter $\varepsilon$. By treating the perturbation parameter as an interval, we are also able to give a description of the orbits of the map in terms of symbolic dynamics for all perturbation values inside the given interval. We chose the standard map for a couple of reasons. First, it is a very well-known system with rich dynamics. Second, in contrast with previous results for H\'enon-like systems (e.g. \cite{DFT, raf-plats}), the standard map is a volume-preserving map which exhibits both hyperbolic and elliptic behavior, which make the task of automating the construction of index pairs considerably more difficult.  Finally, we generalize the methods of~\cite{DFT} to work on manifolds with nontrivial topologies and not just $\mathbb{R}^n$, and our application to the standard map exemplifies this.

In the present constructions, we benefit much from some \emph{a-priori} information of the map. The distinction between a hyperbolic or elliptic periodic orbit, for example, is a crucial one. Symmetries of the map, the knowledge of the precise location of any heteroclinic or homoclinic orbits are also valuable data from which to begin. In our view, a little previous knowledge goes a long way, as we exploit this a-priori information to make the constructions highly efficient while keeping the entire process highly automated.

We mention that besides~\cite{DFT} there have been other recent efforts towards the automation of computer-assisted proofs in dynamical systems. In particular,~\cite{jay1} is a collection of set-oriented algorithms for the creation of index pairs adapted for volume-preserving maps. In addition,~\cite{hungarian1, hungarian2} use the method of covering relations after reducing the problem of finding suitable sets and maps between sets to a problem of global optimization theory.

\rafnote{Broke the following off of the last paragraph, and threw in the top-down / bottom-up stuff.  I think it fits nicely.}

In contrast to other Conley index methods~\cite{DFT,jay1} which are \emph{top-down}, in that they start with a large portion of the phase space and try to narrow down to the invariant objects, our method is \emph{bottom-up}, in that we are guided by the dynamics of the system as we ``grow out'' the index pairs given good numerical approximations of the invariant sets (see Algorithms~\ref{alg:connections} and~\ref{alg:grow-ins}).  The goal of our bottom-up approach is to compute a minimal number of images of a discretized version of the map (see section~\ref{sec:combinatorial}), which is often the most expensive part of computations. Our approach thus achieves high efficiency in the automation of the construction of an index pair.

The present work has been motivated by the challenges brought by the standard map to the task of automating the construction of index pairs and by what we expect will be obstacles for construction of index pairs in future applications. The techniques of this paper along with those of~\cite{DFT} and the references therein provide a solid set of tools for automating computer assisted proofs in the paradigm of planar maps with positive entropy through discrete Conley index theory. These methods are technically not restricted to dimension 2, but remain mostly untested in higher dimensional examples with higher dimensional unstable bundles. To the authors' knowledge, there is an example in~\cite{sarah-kot} with a two dimensional unstable bundle, but no others in the existing computational Conley index literature. Having focused on making the construction of index pairs highly efficient in this paper, it remains to apply them to such higher dimensional examples, which we plan to do in future work.

To the authors' knowledge, the rigorous bounds presented here on the topological entropy for the standard map are the only ones of their kind in the literature. We could only find~\cite{tanikawa} for computational lower bounds of the topological entropy for the standard map at various parameter values, but the results there are non-rigorous. Thus it would be interesting to see how other topological methods (for example \cite{covering1, newhouse})  perform for the same values we have examined in this paper.

This paper is organized as follows: in section 2 we review the necessary background for the subsequent sections. Section 3 details the data structures we use in our algorithms as well as presents the two new algorithms which we consider as the main contribution of this work, Algorithms~\ref{alg:connections} and~\ref{alg:grow-ins}. We apply these algorithms in section 4 to the standard map to give lower bounds for the topological entropy at different values of the perturbative parameter. Theorem~\ref{thm:main} gives lower bounds of the topological entropy of the standard map as given by topological entropy of sub-shifts of finite type, for all $f_\varepsilon$ with $\varepsilon\in[0.7,2.0]$. Theorem~\ref{thm:half} gives a positive lower bound for the topological entropy for $\varepsilon = \frac{1}{2}$. Theorems~\ref{thm:connections2} and~\ref{thm:awesome} give a positive lower bound for the topological entropy for $\varepsilon = 2$ by finding connecting orbits to hyperbolic periodic orbits of higher period. We conclude with some comments and remarks on the implementation in Section \ref{sec:implementation}.
\begin{ack}
We would like to thank Jay James for many helpful discussions during the course of this work, as well as for his help with the implementation of the methods of~\cite{parameterization3} for the standard map, as well as Sarah Day for helpful comments on an early draft of the paper. We also thank the Center for Applied Mathematics at Cornell University for allowing us access to their computational facilities, where the computations in this paper were carried out. We would like to thank an anonymous referee for helping us make the paper clearer.
\end{ack}
\section{Background}
We review in this section the necessary facts about topological entropy, symbolic dynamics, and the discrete Conley index which will be relevant in the later sections. Although the exposition will not be in depth, we encourage the interested reader to consult section 2 of~\cite{DFT} for a more detailed exposition. For deeper treatment of the discrete Conley index, see~\cite{MM:cit}.
\subsection{Topological Entropy and Symbolic Dynamics}
Let $f: X\rightarrow X$ be a map. The \emph{topological entropy} of $f$ is the quantity $h(f)\in\mathbb{R}\cup\{\infty\}$ which is a good measure of the complexity of $f$. A map with positive topological entropy usually exhibits chaotic behavior in many of its trajectories.
\begin{definition}
Let $f:X\rightarrow X$ be a continuous map. A set $W\subset X$ is called $(n,\varepsilon)$-separated for $f$ if for any two different points $x, y\in W$, $\|f^k(x)-f^k(y)\|>\varepsilon$ for some $k$ with $0\leq k < n$. Let $s(n,\varepsilon)$ be the maximum cardinality of any $(n,\varepsilon)$-separated set. Then
$$h(f) = \sup_{\varepsilon>0}\limsup_{n\rightarrow\infty}\frac{\log s(n,\varepsilon)}{n}$$
is the \emph{topological entropy} of $f$.
\end{definition}
Although it is always well-defined, it is usually impossible to compute the topological entropy directly from the definition. The set of systems for which it is computable is rather small, but it contains all systems from symbolic dynamics.

Let $\Sigma_N = \{0,\dots,N-1\}^\mathbb{Z}$ be the set of all bi-infinite sequences on $N$ symbols. It is well-known that $\Sigma_N$ is a complete metric space. Define the \emph{full N-shift} $\sigma:\Sigma_N\rightarrow\Sigma_N$ to be the map acting on $\Sigma_N$ by $(\sigma(x))_i = x_{i+1}$. Given an $N\times N$ matrix $A$ with $A_{i,j}\in\{0,1\}$ we can define a \emph{sub-shift of finite type} defined on $\Sigma_A\subset\Sigma_N$, where $x\in\Sigma_A$ if and only if for $x = (\dots,x_i,x_{i+1},\dots)$, $A_{x_i,x_{i+1}} = 1$ for all $i$. In other words, $\Sigma_A$ consists of all sequences in $\Sigma_N$ with transitions of $\sigma$ allowed by $A$. Viewed as a graph with $N$ vertices and with an edge from vertex $i$ to vertex $j$ if and only if $A_{i,j} = 1$, $\Sigma_A$ can be identified with the set of all infinite paths in such graph.

The topological entropy of a sub-shift of finite type is computed through the largest eigenvalue of its transition matrix $A$.
\begin{theorem}
Let $\sigma:\Sigma_A\rightarrow\Sigma_A$ be a subshift of finite type. Then its topological entropy is
$$h(\sigma) = \log \mbox{sp}(A),$$
where $\mbox{sp}(A)$ denotes the spectral radius of $A$.
\end{theorem}
%This is a well-known theorem and the proof can be found in KH.
In order to study a map $f$ of high complexity, it is sometimes possible to study a subsystem of it through symbolic dynamics. This is done through a semi-conjugacy.
\begin{definition}
Let $f:X\rightarrow X$ and $g:Y\rightarrow Y$ be continuous maps. A \emph{semi-conjugacy} from $f$ to $g$ is a continuous surjection $h:X\rightarrow Y$ with $h\circ f = g\circ h$. We say $f$ is semi-conjugate to $g$ if there exists a semi-conjugacy.
\end{definition}
Note that through a semi-conjugacy $h:X\rightarrow Y$, information about $g$ acting on $Y$ gives information about $f$ acting on $x$. In particular, the topological entropy of $g$ bounds from below the topological entropy of $f$.
\begin{theorem}
Let $f$ be semi-conjugate to $g$. Then $h(f)\geq h(g)$.
\end{theorem}
\begin{corollary}
Let $f$ be semi-conjugate to $\sigma:\Sigma_A\rightarrow\Sigma_A$, for some $N\times N$ matrix $A$. Then
$$h(f) \geq \mbox{sp}(A).$$
\end{corollary}
\subsection{The Discrete Conley Index}
%Conley index theory is a topological tool in dynamical systems which started in the 70's as a refinement of Morse's theory on gradient-like flows. An analogous version has been developed for discrete systems.
Let $f:M\rightarrow M$ be a continuous map, where $M$ is a smooth, orientable manifold.
\begin{definition}
A compact set $K\subset M$ is an \emph{isolating neighborhood} if 
$$\mbox{Inv}(K,f)\subset\mbox{Int}(K),$$
where $\mbox{Inv}(K,f)$ denotes the maximal invariant set of $K$ and $\mbox{Int}(K)$ denotes the interior of $K$. $S$ is an \emph{isolated invariant set} if $S = \mbox{Inv}(K,f)$ for some isolating neighborhood $K$.
\end{definition}
\begin{definition}[~\cite{salamon}]
Let $S$ be an isolated invariant set for $f$. Then $P = (P_1,P_0)$ is an \emph{index pair} for $S$ if
\begin{enumerate}
\item $\overline{P_1\backslash P_0}$ is an isolating neighborhood for $S$.
\item The induced map 
\[
  f_P(x)=\left\{
    \begin{array}{ll}
      f(x) & \mbox{ if $x$, $f(x)\in P_1/P_0$ },\\
      \mbox{$[P_0]$} & \mbox{ otherwise}.
    \end{array}\right.
\]
defined on the pointed space $(P_1/P_0,[P_0])$ is continuous.
\end{enumerate}
\end{definition}
\begin{definition}[~\cite{BobWilliams}]
Let $G, H$ be abelian groups and $\varphi:G\rightarrow G$, $\psi:H\rightarrow H$ homomorphisms. Then $\varphi$ and $\psi$ are \emph{shift equivalent} if there exist homomorphisms $r:G\rightarrow H$ and $s:H\rightarrow G$ and a constant $k\in\mathbb{N}$ such that
$$r\circ\varphi = \psi\circ r,\;\; s\circ\psi = \varphi\circ s,\;\; r\circ s = \psi^k,\;\;\mbox{ and }s\circ r = \varphi^k.$$ 
\end{definition}
Shift equivalence defines an equivalence relation, and we denote by $[\varphi]_s$ the class of all homomorphisms which are shift equivalent to $\varphi$.
\begin{definition}[~\cite{FranksRicheson}]
Let $P = (P_1,P_0)$ be an index pair for an isolated invariant set $S = \mbox{Inv}(\overline{P_1\backslash P_0}, f)$ and let $f_{P*}:H_*(P_1,P_0;\mathbb{Z})\rightarrow H_*(P_1,P_0
;\mathbb{Z})$ be the map induced by $f_P$ on the relative homology groups $H_*(P_1,P_0;\mathbb{Z})$. The \emph{Conley index} of $S$ is the shift equivalence class $[f_{P*}]_s$ of $f_{P*}$.
\end{definition}
One can think of two maps $f_{P*}$ and $g_{\bar{P}*}$ as being in the same shift equivalence class if and only if they have the same assymptotic behavior. Since an index pair for an isolated invariant set is not unique, the Conley index of and isolated invariant set does not depend on the choice of index pair. There are two elementary yet important results which link the Conley index of an isolated invariant set with the dynamics of $f$. The first one is the so-called Wa\.zewski property:
\begin{theorem}
If $[f_{P*}]_s\neq [0]_s$, then $S\neq\varnothing$.
\end{theorem}
Since similar assymptotic behavior relates two different maps in the same shift equivalence class, it is sufficient then to have a map $f_{P*}$ not be nilpotent in order to have a non-empty isolated invariant set. In practice, non-nilpotency can be verified by taking iterates of a representative of $[f_{P*}]_s$ until non-nilpotent behavior is detected. We also have an elementary detection method of periodic points.
\begin{theorem}[Lefshetz Fixed Point Theorem]
Let $f_*$ be a representative of $[f_{P*}]_s$ with maps $f_k:H_k(P_1,P_0;\mathbb{Z})\rightarrow H_k(P_1,P_0;\mathbb{Z})$ which are represented by matrices. Then if
$$\Lambda(f_*) = \sum_{k\geq 0}(-1)^k\mbox{ tr }f_k\neq 0,$$
then $f$ has a fixed point. Moreover, if $\Lambda(f_*^n)\neq 0$, then $f$ has a periodic point of period $n$.
\end{theorem}
Since traces are preserved under shift equivalence, $\Lambda(f_*)$ is independent of the representative of $[f_{P*}]_s$, and we may even denote it as $\Lambda([f_{P*}]_s)$. 
\begin{corollary}
\label{cor:fixed}
Let $K\subset M$ be the finite union of disjoint, compact sets $K_1,\dots, K_m$ and let $S = \mbox{Inv}(K,f)$. Let $S' = \mbox{Inv}(K_1,f_{K_m}\circ\cdots\circ f_{K_1})\subset S$ where $f_{K_i}$ denotes the restriction of $f$ to $K_i$. If
$$\Lambda([f_{K_m}\circ\cdots\circ f_{K_1}]_s)\neq 0$$
then $f_{K_m}\circ\cdots\circ f_{K_1}$ contains a fixed point in $S'$ which corresponds to a periodic orbit of $f$ which travels through $K_1,\dots, K_m$ in such order.
\end{corollary}
This is a strong and useful tool for proving the existence of periodic orbits. However, we may have $\Lambda([f_{K_m}\circ\cdots\circ f_{K_1}]_s)= 0$ while $(f_{K_m}\circ\cdots\circ f_{K_1})_*$ is not nilpotent and thus have an invariant set which may be of interest since it behaves like a periodic orbit. So we have an analogous result based on the Wa\.zewski property.
\begin{corollary}
\label{cor:Waz}
Let $K\subset M$ be the finite union of disjoint, compact sets $K_1,\dots, K_m$ and let $S = \mbox{Inv}(K,f)$. Let $S' = \mbox{Inv}(K_1,f_{K_m}\circ\cdots\circ f_{K_1})\subset S$ where $f_{K_i}$ denotes the restriction of $f$ to $K_i$. If
$$[(f_{K_m}\circ\cdots\circ f_{K_1})_*]_s \neq [0]_s,$$
then $S'$ is nonempty. Moreover, there is a point in $S$ whose trajectory visits the sets $K_i$ in such order.
\end{corollary}
Both corollaries can be useful in different settings when used to to prove symbolic dynamics through computer-assistance. When existence of periodic points is of primary concern, one can use Corollary~\ref{cor:fixed}. When existence of trajectories which shadow a certain prescribed path given by a symbol sequence, but that do not necessarily correspond to periodic orbits, Corollary~\ref{cor:Waz} is the better tool. This may occur when there are higher-dimensional invariant sets to which the restriction of the dynamics is quasi-periodic. In either case, what makes the implementation possible is the ease of computability of the traces of the induced maps on homology for Corollary~\ref{cor:fixed} and a sufficient condition for non-nilpotency in the case of Corollary~\ref{cor:Waz}. For more details on the implementation, see~\cite{DFT}.

\subsection{Combinatorial Structures}
\label{sec:combinatorial}
All concepts from the discrete Conley index theory from the previous section have analogous definitions in a combinatorial setting for which computer algorithms can be written.
\begin{definition}
A \emph{multivalued map} $F: X\rightrightarrows X$ is a map from $X$ to its power set, that is, $F(x)\subset X$. If for some continuous, single-valued map $f$ we have $f(x)\in F(x)$ and $F$ is acyclic, then $F$ is an \emph{enclosure} of $f$.
\end{definition}
The reason multivalued maps and enclosures are used in our computations is that if they are done properly, they give rigorous results. Moreover, if $F$ is an enclosure of $f$ and $(P_0,P_1)$ is an index pair for $F$ (as we will define below), then it is an index pair for $f$. It also follows that if we can compute the Conley index of $F$ and process the information encoded in it, we may obtain information about the dynamics of $f$.

We begin by setting up a grid $\mathcal{G}$ on $M$, which is a compact subset of the $n$ dimentional manifold $M$ composed of finitely many elements $\mathcal{B}_i$. Each element is a cubical complex, hence a compact set, and it is essentially an element of a finite partition of a compact subset of $M$. In practice, all elements of the grid are rectangles represented as products of intervals (viewed in some nice coordinate chart), that is, for $\mathcal{B}_i\in\mathcal{G}$, $\mathcal{B}_i = \prod_{k=1}^n[x^i_k,y^i_k]$. We refer to each element of $\mathcal{G}$ as a \emph{box} and each box is defined by its center and radius, i.e., $\mathcal{B}_i = (c_i,r_i)$, where $c_i$ and $r_i$ are $n$-vectors with entries corresponding to the center and radius, respectively, in each coordinate direction. In practice we have $r_i$ the same for all boxes in $\mathcal{G}$ but this is in no way necessary and there may be systems for which variable radii for boxes provide a significant advantage. For a collection of boxes $\mathcal{K}\subset\mathcal{G}$, we denote by $|\mathcal{K}|$ its \emph{topological realization}, that is, its corresponding subset of $M$. From now on we will use caligraphy capital letters to denote collections of boxes in $\mathcal{G}$ and by regular capital letters we will denote their topological realization, e.g., $|\mathcal{B}_i| = B_i$.

The way we create a grid is as follows. We begin with with one big box $\mathcal{B}$ such that $|\mathcal{B}|$ encloses the area we wish to study. Then we subdivide $\mathcal{B}$ $d$ times in each coordinate direction in order to increase the resolution at which the dynamics are studied. The integer $d$ will be refered to as the \emph{depth}. Thus working at depth $d$ gives us a maximum of $2^{dn}$ ($n=\dim M$) boxes with which to work of each coordinate of size $2^{-d}$ relative to the original size of the box $\mathcal{B}$.
\begin{definition}
A \emph{combinatorial enclosure} of $f$ is a multivalued map $\mathcal{F}:\mathcal{G}\rightrightarrows \mathcal{G}$ defined by
$$\mathcal{F}(\mathcal{B}) = \{\mathcal{B}'\in\mathcal{G}:|\mathcal{B}'|\cap F(B)\neq\varnothing\},$$
where $F$ is an enclosure of $f$.
\end{definition}
In practice, combinatorial enclosures are created as follows. One begins with $\mathcal{B}\in\mathcal{G}$ and defines $F(x)$, $x\in B=|\mathcal{B}|$, as the image of $B$ using a rigorous enclosure for the map $f$. Rigorous enclosures are obtained by keeping track of the error terms in the computations of the image of a box and making sure the true image $f(B)$ is contained in $|\mathcal{F}(\mathcal{B})|$. In this paper, all rigorous enclosures will be obtained with the use of interval arithmetic using Intlab~\cite{intlab}. Note that $|\mathcal{F}|$ becomes an enclosure of $f$. $\mathcal{F}:\mathcal{G}\rightrightarrows\mathcal{G}$ can also be represented as a matrix $\mathcal{T}$ with entries in $\{0,1\}$ with $\mathcal{T}_{i,j} = 1$ if and only if $\mathcal{B}_j \in\mathcal{F}(\mathcal{B}_i)$. Moreover $\mathcal{T}$ can be viewed as a directed graph with vertices corresponding to individual boxes in $\mathcal{G}$ and edges going from box $i$ to box $j$ if and only if $\mathcal{T}_{i,j}=1$.
\begin{definition}
A \emph{combinatorial trajectory} of a combinatorial enclosure $\mathcal{F}$ through $\mathcal{B}\in \mathcal{G}$ is a bi-infinite sequence $\gamma_G = (\dots,\mathcal{B}_{-1},\mathcal{B}_0,\mathcal{B}_1,\dots)$ with $\mathcal{B}_0=\mathcal{B}$, $\mathcal{B}_n\in \mathcal{G}$, and $\mathcal{B}_{n+1}\in \mathcal{F}(\mathcal{B}_n)$ for all $n\in \mathbb{Z}$.
\end{definition}
The definitions which follow are by now standard in the computational Conley index literature. We will state definitions and refer the reader to~\cite{DFT} for the algorithms which construct the objects defined.
\begin{definition}
The \emph{combinatorial invariant set} in $\mathcal{N}\subset \mathcal{G}$ for a combinatorial enclosure $\mathcal{F}$ is 
$$\mbox{Inv}(\mathcal{N},\mathcal{F}) = \{ \mathcal{B}\in\mathcal{G}:\mbox{ there exits a trajectory }\gamma_G\subset\mathcal{N}\}.$$
\end{definition}
\begin{definition}
The \emph{combinatorial neighborhood} or \emph{one-box beighborhood} of $\mathcal{B}\subset\mathcal{G}$ is
$$o(\mathcal{B}) = \{\mathcal{B}'\in\mathcal{G}:|\mathcal{B}'|\cap|\mathcal{B}|\neq\varnothing\}.$$
\end{definition}
\begin{definition}
If
$$o(\mbox{Inv}(\mathcal{N},\mathcal{F}))\subset \mathcal{N}$$
then $\mathcal{N}\subset\mathcal{G}$ is a \emph{combinatorial isolating neighborhood} for $\mathcal{F}$.
\end{definition}
A procedure for creating a combinatorial isolating neighborhood is discussed in Section 3 and an given by Algorithm~\ref{alg:grow-iso}. Once we have a combinatorial isolating neighborhood, it is possible to define and create a combinatorial index pair.
\begin{definition}
A pair $\mathcal{P} = (\mathcal{P}_1,\mathcal{P}_0)\subset\mathcal{G}$ is a \emph{combinatorial index pair} for the combinatorial enclosure $\mathcal{F}$ if its topological realization $P_i = |\mathcal{P}_i|$ is an index pair for any map $f$ for which $\mathcal{F}$ is an enclosure.
\end{definition}
One of the main goals of this paper is a procedure to efficiently compute combinatorial index pairs. This is the content of the next section. We now have made all definitions necessary to define the Conley index. Computing the induced map on homology at the combinatorial level is no trivial task. We refer the enthusiastic reader to~\cite{CompHom} for a very thourough exposition on computing induced maps on homology. In practice, we use the computational package \emph{homcubes}, part of the computational package \emph{CHomP}~\cite{chomp}, which computes the necessary maps on homology to define the Conley index.
\section{Automated Symbolic Dynamics}
The results of~\cite{DFT} show that it is possible to create an automated procedure to rigorously prove a semi-conjugacy from a map $f:M\rightarrow M$ to a subshift $\Sigma_A$. The procedure detailed in~\cite{DFT} can be summarized as follows:
\begin{enumerate}[(i)]
\item Start with a rectangle in $\mathbb{R}^n$ and obtain a grid of
  boxes of constant radius by partitioning it in each coordinate direction $d$ times for some fixed $d$.
\item Compute the combinatorial enclosure $\mathcal{F}$ of $f$ in a
  neighborhood of the area of interest, given by the interval
  arithmetic images of boxes in the grid. For some fixed $k$, create a collection
  $\mathcal{P}_k\subset\mathcal{G}$ of boxes which are periodic of
  period $n\leq k$ under $\mathcal{F}$ by finding non-zero entries of
  the diagonal of $\mathcal{T}^n$.
\item Using shortest path algorithms, and denoting by
  $\mathcal{D}_{ij}$ any shortest path between
  $\mathcal{B}_i\in\mathcal{P}_k$ and $\mathcal{B}_j\in\mathcal{P}_k$
  in $\mathcal{G}$ (if there is no such path, $\mathcal{D}_{ij} =
  \varnothing$), create a collection $\mathcal{A}\subset\mathcal{G}$
  by $\mathcal{A} = \mathcal{P}_k \cup \left(\bigcup_{i\neq j}
    \mathcal{D}_{ij}\right)$. Using the appropriate algorithms, create
  a combinatorial isolating neighborhood for $\mathcal{A}$, a
  combinatorial index pair, and compute its Conley index $[f_{P*}]_s$.
\item Working with a representative of $[f_{P*}]_s$ as matrices (one
  for each level of homology) over the integers, we filter out all the
  nilpotent behavior and find a smaller representative which exhibits
  only recurrent behavior.
\item We study $[f_{P*}]_s$ by its action on the generators of
  $H_1(P_1,P_0;\mathbb{Z})$ grouped by disjoint components of
  $P_1\backslash P_0$, and using Corollary~\ref{cor:Waz} we perform a
  finite number of calculations to prove a semiconjugacy from $f$ to
  $\Sigma_{A}$, where $A$ is a $m\times m$ matrix, and $m$ is a number
  less than or equal to the number of disjoint components of
  $P_1\backslash P_0$.
\end{enumerate}
This method was applied to the H\'enon map and a semi-conjugacy to a
subshift on 247 symbols was proved, which gave a rigorous lower bound
on the topological entropy.

We make a few remarks about the approach. First, although all periodic orbits of the H\'enon map exhibit hyperbolic behavior, this is not true in all systems. In Hamiltonian systems, for example, one expects roughly half of the periodic orbits to be isolated invariant sets. Thus looking at the diagonal of $\mathcal{T}^n$ is not enough to capture orbits which will give us isolated invariant sets and a nontrivial Conley index. Another issue which arises in practice is the computational complexity of the algorithms to prove a semiconjugacy to a subshift. The complexity increases exponentially with the dimension of the system and if we have any hope of getting results in systems of dimension higher than 2, we must perform the computations as efficiently as possible.

With this in mind, we view the entire process of proving a semi-conjugacy as the composition of two major steps:
\begin{enumerate}
\item The gathering of recurrent, isolated invariant sets. This can be done through a combination of non-rigorous numerical methods, graph algorithms, set-oriented methods, linear algebra operations, et cetera.
\item A proof that this behavior in fact exists through a semi-conjugacy to a subshift.
\end{enumerate}

We point out that the second step is solved in~\cite{DFT}. That is,
given as input an index pair $(P_0,P_1)$ and the combinatorial
enclosure $\mathcal{F}$ of $f$ used to create it, \emph{proving
  semi-conjugacy from $f$ to a subshift is completely automated}.  The
first step is also largely dealt with, but the approach is not as
general as it can be, as it was devised to deal with maps like the
H\'{e}non map where isolation is easy. We wish to consider cases where
one cannot simply compute the map on all boxes in the area of
interest, either because the dimension of the attractor is too high or
because one needs a very high box resolution to achieve
isolation. Whereas~\cite{DFT} focused on (2), here we focus on (1) and
on how to produce index pairs more efficiently.  In settings where
isolation is difficult, our methods in this paper to efficiently deal
with (1) and the solution in~\cite{DFT} of (2) constitute a better
method to prove semi-conjugacies to a SFT.

We provide an efficient method of computing an index pair in step (1)
when one knows roughly what to look for. In essence, we believe that a
small amount of \emph{a-priori} knowledge goes a long way. In other
words, we can simplify and make computations much more efficient if we
know something about dynamics beforehand, like previous knowledge of
the location of periodic, heteroclinic, homoclinic orbits, special
symmetries of the system, et cetera. We will illustrate this approach
with examples in Section~\ref{sec:computations}.

More precisely, we take as input a list of periodic points and pairs of
points which might have connecting orbits between them, and produce an
index pair containing all of the points and connections, if possible.
For example, if one had numerical approximations of two fixed points
and conjectured that there was a connecting orbit between them, our
algorithms could potentially prove the existence of the connection.

We break the task of finding an index pair into two steps.  The first
is to approximate numerically the invariant objects whose existence we
want to prove, and the second is to cover such an approximation with
boxes and add boxes until the index pair conditions are satisfied.
Before describing these steps, however, we introduce our underlying
data structures.

\subsection{Data Structures}
\label{sec:datastructures}
The algorithms we present rely on some basic data structures to encode
the topology of the phase space and the behavior of the map.  We
require two routines to keep track of the topology: finding the box
corresponding to a particular point in the phase space (to compute the
multivalued map on boxes), and determining which boxes are adjacent
(to determine isolation).  We then of course need a method of storing
the multivalued map itself.

To accomplish all of this, we consider a subset $S$ of our grid and
enumerate the boxes in $S$, giving each a unique integer.  We then
store the position information in a binary search tree, which we
simply call \emph{the tree}, so that the index $i$ for the box $b_i$
covering a given point $x$ in the phase space (i.e., $x\in |b_i|$) may
be computed efficiently.  We use the \emph{GAIO} implementation for
our basic tree data structure~\cite{GAIO} only, and make use of some
enhancements, discussed below.  To encode the topology, we store the
adjacency information in a (symmetric) binary matrix $Adj$, where
$Adj_{\,ij} = Adj_{\,ji} = 1$ if and only if boxes $b_i$ and $b_j$ are
adjacent.  More precisely, $Adj_{\,ij}=1$ if and only if $|b_i| \cap
|b_j| \neq \varnothing$.  We will call this matrix the \emph{adjacency
  matrix}. This is a crucial component when working on spaces with
nontrivial topology such as the torus, as we will see in section
\ref{sec:computations}. Finally, we store the multivalued map as a
transition matrix $P$ such that $P_{ji} = 1$ if and only if $b_j\in
\mathcal{F}(b_i)$.

Below are the operations of the tree data structure:
\begin{enumerate}[(i) ]
\item {\bf $S$ = boxnums() } -- Return a list of all box numbers in the tree
\item {\bf $S$ = find($C$) } -- Given a set $C\subset M$ of points, return the
  indices $S$ of boxes covering any $x\in C$ %-- $O(\log B)$
\item {\bf insert($C$) } -- Insert boxes into the tree covering points
  $C\subset M$ (if they do not already exist) % -- $O(B\log B)$
\item {\bf delete($S$) } -- Remove boxes with indices in $S$ from the
  tree % -- $O(B\log B)$
\item {\bf subdivide() } -- For each box $b$ in the tree, divide each of its
  coordinate directions in half creating $2^n$ smaller boxes, thus
  increasing the depth of the tree by 1 % -- $O(B)$
\end{enumerate}

Although the \emph{GAIO} tree data structure is in many ways
well-suited to our task, it has a few simple drawbacks which
significantly restrict our computations.  Specifically, the operations
insert($\cdot$) and delete($\cdot$) scramble the box numbers of the
tree as a side-effect.  Thus, if we want to keep track of our current
box set while inserting or deleting boxes, we have to spend extra time
computing the new box numbers.  Without knowledge of how the
renumbering is done, this would take $O(n \log n)$ time, where $n$ is
the number of boxes, since for each box we have to search the tree to
find its new number.  The best one could hope for would be $O(\log n)$
with careful bookkepping (one still needs logarithmic time to locate
the place in the tree for insertion or deletion).  We managed to
compute the new box numbers in $O(n)$ time by observing that the
numbering is given by a deterministic depth-first search (DFS)
traversal of the tree from the root box, and backsolving the modified
numbers accordingly.

We make use of several other subroutines as well:
\begin{itemize}
\item {\bf $P$ = transition\_matrix(tree,$f$[,$S$])} -- Compute the
  multivalued map $\mathcal{F}$ using $f$ (optionally only for boxes
  in $S$) and return it as a matrix
\item {\bf $Adj$ = adjacency\_matrix(tree)} -- Using the tree.find($\dots$)
  function, compute the adjacency matrix which describes the topology
  of the boxes in the tree
\item {\bf $oS$ = tree\_onebox($S$,$Adj$)} -- The indices of boxes in a
  one-box neighborhood of $S$ in the tree
\item {\bf tree = insert\_onebox(tree,$S$)} -- Using tree.insert($\dots$),
  add boxes to the tree that would neighbor a box in $S$ were they
  already in the tree
\item {\bf tree = insert\_image(tree,$f$,$S$)} -- Compute the images of each
  box in $S$ under $f$ using interval arithmetic
\item {\bf $S$ = grow\_isolating($S$,$P$,$Adj$)} -- Compute an isolating
  neighborhood of $S$ in the tree.  This procedure is called
  grow\_isolating\_neighborhood in~\cite{DFT}; for completeness we
  restate it here as Algorithm~\ref{alg:grow-iso}.
\item {\bf $S$ = invariant\_set($N$,$P$)} -- Compute the maximal invariant
  set containing $N$ according to the multivalued map $P$
\end{itemize}

\begin{algorithm}[!ht]
\caption{$grow\_isolating$: Growing an isolating neighborhood}
\begin{algorithmic}
  \STATE Input: $S$,$P$,$Adj$
  \LOOP
  \STATE $I$ = invariant\_set($S$,$P$)                   
  \STATE $oI$ = tree\_onebox($I$,$Adj$)                  
  \STATE \textbf{if} $oI \subseteq S$ \textbf{return} $I$
  \STATE $S = oI$                                        
  \ENDLOOP
  \STATE Output: $S$
\end{algorithmic}
\label{alg:grow-iso}
\end{algorithm}

Note the difference between insert\_onebox(tree,$S$) and
tree\_onebox($S$,$Adj$); the former adds all boxes to the tree that
touch $S$, but the latter find boxes that touch which are already in
the tree.

\subsection{Numerical Approximations}
\label{sec:approx}

We assume here that one has numerical approximations of points which
are part of hyperbolic, invariant sets such as periodic points and
homoclinic points.  The goal then is to find the desired connecting
orbits between specified points.  More formally, given a finite set $X
= \{(x_i,y_i)\}_i \subset M\times M$ of pairs of points in the phase
space, we want to find a small set of boxes $S$ at depth $d$ such that
$\bigcup X \subset |S|$ and for every pair $p_i$, if $x_i \in |b_x|$
and $y_i \in |b_y|$, then there is a path from $b_x$ to $b_y$
according to the multivalued map $\mathcal F$.  In other words, we
wish to find connections between each $x_i$ and $y_i$ at depth $d$.
Intuitively, we think of the set $X$ as the ``skeleton'' of the
invariant behavior of interest.  For example, if we wanted to connect
a period 2 orbit ${a,b}$ to a fixed point ${c}$ and back, we could
have $X = \{(a,c),(b,c),(c,b),(c,a)\}$.  We assume of course that the
desired connections exist.

Typically the most expensive part of the calculations, especially in
applications involving rigorous enclosures, is the box image
calculations.  Thus we wish to find an algorithm which computes as few
images as possible, but also has a reasonable running time in terms of
the other parameters.

The first algorithm one might try is to go to depth $d$, add all boxes
in the general area of the invariant objects, and then use shortest
path algorithms to find the connections.  Unfortunately this requires
computing $O(2^{dn})$ image calculations, where $n=\dim(M)$.  We can
improve this bound by instead doing a breadth-first search at depth
$d$, that is, starting at each $x_i$ we compute images until we reach
$y_i$ in a breadth-first manner.  This gives roughly $O(|X| b^\ell)$,
where $b$ is the average image size of a box, which will typically
depend exponentially on $n$, and $\ell$ is the average connection
length.  This is no longer exponential in $d$, and $|X|$ is typically
relatively small, but there are still many image calculations as $b$
is often very large.

Fortunately we can achieve only $O(d|X|\ell)$ image calculations using
Algorithm~\ref{alg:connections}, which is essentially an recursive
version of the first algorithm: we compute the connections at a low
depth, then subdivide to get to the next depth and grow one-box
neighborhoods until the connections are found again, and repeat until
we reach the final depth.

\begin{algorithm}[!ht]
\caption{The connection insertion algorithm}
\begin{algorithmic}
  \STATE Input: $f$, $X = \{(x_i,y_i)\}_i\subset M\times M$
  \FOR{depth = $d_\text{start}$ \textbf{to} $d_\text{end}$}
  \STATE tree.insert($\cup_i\{x_i,y_i\}$)
  \LOOP
  \FORALL{$i$}
  \STATE $p_i$ = shortest path from $x_i$ to $y_i$ in $P$
  (or $\varnothing$ if no path)
  \ENDFOR
  \STATE \textbf{if} $\;\forall i \; p_i \neq \varnothing$, \textbf{break loop}
  \STATE tree = insert\_onebox(tree,tree.boxnums())
  \STATE $P$ = transition\_matrix(tree,$f$)
  \ENDLOOP
  \STATE tree.delete(tree.boxnums() $\setminus \; \cup_i p_i$)
  \STATE tree.subdivide()
  \STATE $P$ = transition\_matrix(tree,$f$)
  \ENDFOR
\end{algorithmic}
\label{alg:connections}
\end{algorithm}

By a strict reading of Algorithm~\ref{alg:connections}, we might end
up computing many box images repeatedly from the
transition\_matrix($\cdot$) call in the inner loop.  To avoid this, we
simply cache box image calculations: each time a box image calculation
is called for, we first check to see whether we have computed it
already; if so we look it up, and if not we compute it and store it.
With caching, we achieve the $O(d|X|\ell)$ bound on image calculations.

Note that a box connection found at a depth $d$ may only correspond to
a $\epsilon_0 2^{-d}$-chain rather than an actual orbit, where
$\epsilon_0$ is the length of a box diagonal at depth 0. 
 Typically we rely on shortest path algorithms
to give us the most plausible orbits, but there are cases, such as
spiralling behavior and other `roundabout' connections, where the
shortest paths will not be true orbits. For example,
even the identity map has a path from any box to another according to
the transition matrix. In all of these cases it may
be necessary to use other methods for approximating or connecting
orbits, such as computing the stable and unstable manifolds of the
hyperbolic invariant sets to high accuracy and computing heteroclinic
intersections.

Roughly speaking, a long connection between $p$ and $q$ will have many
of these $\epsilon_0 2^{-d}$-chains from $p$ to $q$ to compete with,
the vast majority of which do not correspond to actual trajectories.
In fact, in many cases the chains will actually be shorter than the
true connection.  For example, in a simple slow rotation about the
origin $(r,\theta)\mapsto(r,\theta+2\pi/k)$, every point is either
period $k$ or $1$ (the origin), but one would find closed chains of length
much lower than $k$ when sufficiently close to the origin.  The result of this
observation is that longer connections will require more iterations of
the insert\_onebox(tree,tree.boxnums()) call, since the probability of
picking a $\epsilon_0 2^{-d}$-chain which corresponds to an actual
orbit becomes quite small as the connection length grows.
Consequently, longer connections will take significantly longer to
compute than shorter ones using Algorithm~\ref{alg:connections}.

One way around this issue is to make use of the duality between
finding connections between periodic orbits of low period and finding
periodic orbits of higher period.  On the one hand, by finding enough
connecting orbits (enough to capture horseshoe dynamics) between
periodic orbits of low period, one finds infinitely many periodic
orbits of higher periods.  On the other hand, it is often the case
that by finding enough periodic orbits of all periods (up to some high
period) at a given depth, most of the connecting orbits which live
close to them should also be captured.  Since dealing with long
connections can be problematic, it is simpler to deal with specific, high
period, periodic orbits if they are easily obtainable.

A second way to compute long connections is to exploit some a priori
knowledge of the map.  We apply both of these techniques in
section~\ref{sec:computations}.

\subsection{Growing an Index Pair}
\label{sec:growing}

Given a small starting set $S$ of boxes corresponding to numerical
guesses of hyperbolic invariant sets, we now wish to create an index
pair from $S$.  To do this, we compute an isolating neighborhood $N$
of $S$.  A first approach might be to use the grow\_isolating
algorithm, Algorithm~\ref{alg:grow-iso}, on the whole of $M$, but as
we are concerned with the setting where isolation is difficult, this
will simply fail to find the desired structures.  Even if we manage to
cover only what we are interested in, this top-down approach requires
computing many images; as with Algorithm~\ref{alg:connections}, we
wish to minimize the number of image calculations while keeping a
reasonable running time.  Specifically, we would like an algorithm
that computes the images of $N$ and no other images, which would be
optimal in our setting.  We will see that Algorithm~\ref{alg:grow-ins}
does precisely that.

\begin{algorithm}
  \caption{Growing an isolating neighborhood by inserting boxes}
\begin{algorithmic}[!ht]
  \STATE Input: tree,$f$
  \STATE $B$ = tree.boxnums()
  \STATE $S = B$
  \REPEAT
  \STATE $P$ = transition\_matrix(tree,$f$,$S$)
  \STATE $Adj$ = adjacency\_matrix(tree)
  \STATE $I$ = invariant\_set($S$,$P$)
  \STATE $oI$ = tree\_onebox($I$,$Adj$)
  \STATE achieved\_isolation = $(oI \subseteq S)$
  \STATE $S = oI$                                        
  \STATE tree = insert\_onebox(tree,$I$)
  \STATE tree = insert\_image(tree,$f$,$oI$)
  \STATE $B$ = tree.boxnums() $\setminus\; B$
  \UNTIL {(achieved\_isolation \textbf{and} $B == \varnothing$)}
  \STATE Output: $S$
\end{algorithmic}
\label{alg:grow-ins}
\end{algorithm}

Note that the loop invariant (and thus correctness) of this algorithm
is highly sensitive to the order of the steps.  As before, we cache
box image computations.

We can think of Algorithm~\ref{alg:grow-ins} as a modification of
grow\_isolating (Algorithm~\ref{alg:grow-iso}), where we make use of
caching and \emph{lazy evaluation}, meaning we only add boxes to the
tree and compute box images when we absolutely have to.  This way we
can start with only our initial skeleton of points in the tree and
grow both the tree and multivalued map just enough to accommodate the
isolating neighborhood and verify its isolation.

In fact we can see that we compute exactly the images we need, in a
precise sense.  Consider the returned set $S$ which,
\emph{a-posteriori}, must be equal to tree\_onebox($I$,$Adj$) for some
invariant set $I$.  Since we terminated, $S$ must be the true one-box
neighborhood of $I$; otherwise, we would have added new boxes in the
insert\_onebox call and failed to terminate.  Similarly, the image of
$S=o(I)$ must already be in the tree, since otherwise the
insert\_image call would have resulted in new boxes.  Thus, $S$ is an
isolating neighborhood of $I$, and since we grew $S$ only by adding
boxes to it, we have computed images only for boxes in $S$, which is
precisely what we need to verify its isolation.

Algorithm~\ref{alg:grow-ins} is the cornerstone of our approach in
this paper.  Without it, no matter how clever our numerical
approximations are, producing index pairs would essentially be just as
expensive as computing the map on all boxes.  This is especially
important when one is working with a low-dimensional invariant set in
a high-dimensional embedded space, or when the hyperbolic, invariant
sets whose existence we wish to prove are tightly squeezed between
non-hyperoblic sets or even singularities of the map.

To see this more precisely, consider the memory required to store the
box images using our bottom-up insertion approach as compared to the
top-down approach of \cite{DFT} ; let $M_{\text{ins}}$ and
$M_{\text{DFT}}$ be the memory in each case.  As above let $n$ be the
dimension of the manifold, and $K_d$ be the number of boxes needed to
cover the invariant set at depth $d$ (which is independent of the
method used).  In settings where isolation is difficult, the method
of~\cite{DFT} would typically require $M_{\text{DFT}} = O(2^{dn})$.
Using Algorithms~\ref{alg:connections} and~\ref{alg:grow-ins},
however, we can achieve $M_{\text{ins}} = K_d$.  Thus, in situations
where $K_d \ll 2^{dn}$, we get a tremendous savings using our
bottom-up method, and if the depth required for isolation is high,
this savings could be the difference between infeasible and feasible.
We will see a concrete example of this in the next section when we
study the standard map, with further discussion in
section~\ref{sec:implementation}.

\section{Computations}
\label{sec:computations}
We apply our methods to the standard map
\begin{equation*}
  f_\varepsilon:(x,y)\mapsto
  \left(x+y+\frac{\varepsilon}{2\pi}\sin(2\pi x)\mbox{ mod 1},
    \; y+\frac{\varepsilon}{2\pi}\sin(2\pi x)\mbox{ mod 1}\right),
  \label{eq:stdmap}
\end{equation*}
where $\varepsilon>0$ is a perturbation parameter. This map is perhaps
the best-known exact, area-preserving symplectic twist map. For
$\varepsilon = 0$ every circle $\{ y = \mbox{constant} \}$ is
invariant and the dynamics of the map consist rotations of this circle
with frequency $y$.

\begin{figure}[t]
\begin{center}
  \includegraphics[width = 4 in]{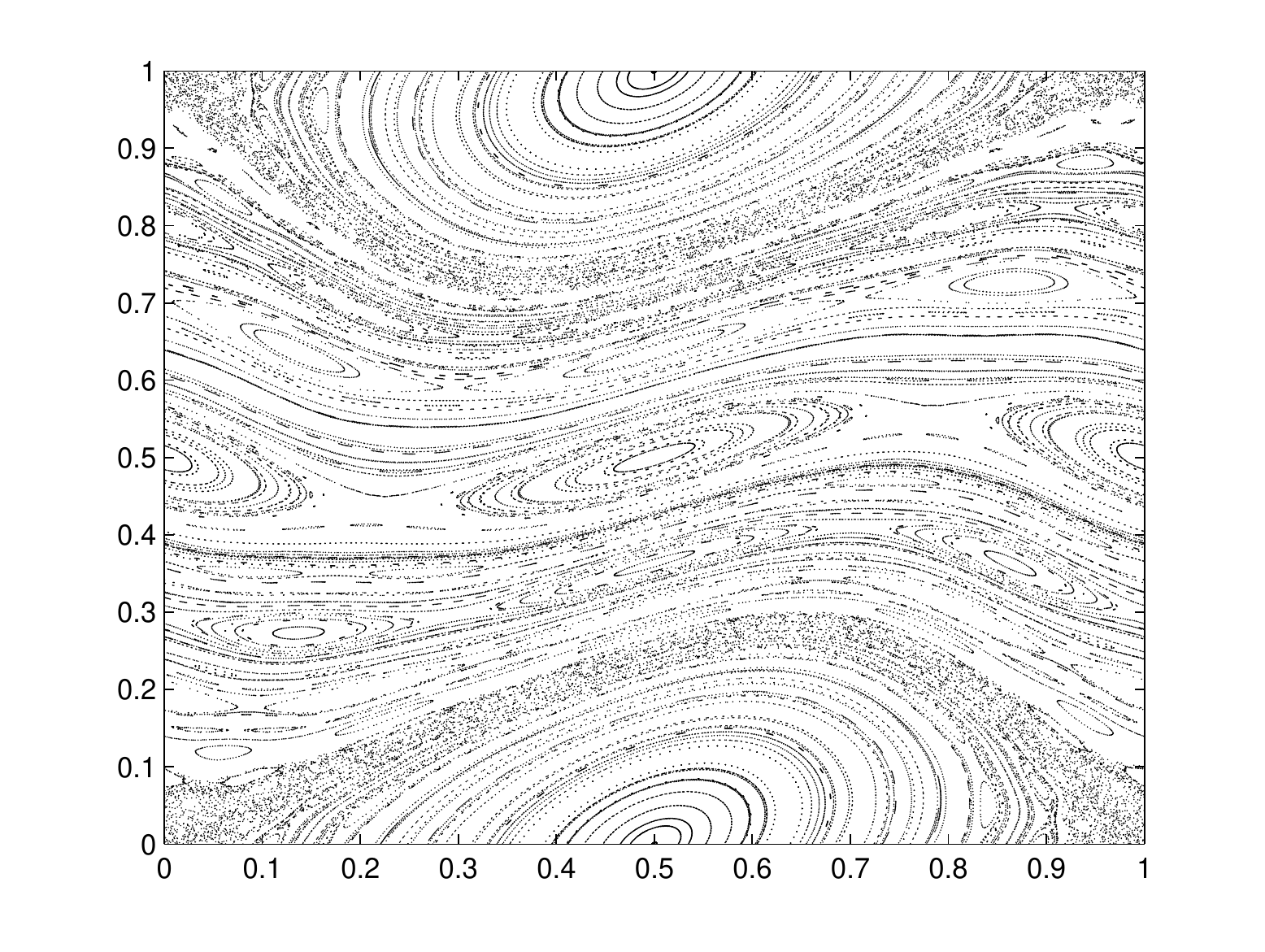}
\vspace{-7mm}
\caption{Plot of trajectories of the standard map for $\varepsilon = \frac{3}{4}$.}
\label{fig:orbits}
\vspace{-3mm}
\end{center}
\end{figure}

The map $f_\varepsilon$ can be recovered by its generating function
$h:\mathbb{T}^2\rightarrow \mathbb{R}$ in the sense that, if
$f_\varepsilon (x,y) = (X,Y)$, then $YdX - ydx = dh(x,X)$. Moreover,
we have its \emph{action} $\mathcal{A}$ which, for a sequence of
points $\{x_N,\dots,x_M\}$, is
\begin{equation}
\label{eqn:action}
\mathcal{A}(x_N,\dots, x_M) = \sum_{k=N}^{M-1} h(x_k,x_{k+1}),
\end{equation}
such that trajectories of $f_\varepsilon$ ``minimize'' the action for
fixed endpoints $\{x_N,x_M\}$, in an analogous way to classical
Lagrangian mechanics (see~\cite[\S 2.5]{gole}).

\begin{figure}[t]
\begin{center}
  \hspace{-50pt} %center the graphic
  \includegraphics[trim = 15 15 415 595, width = 3 in]{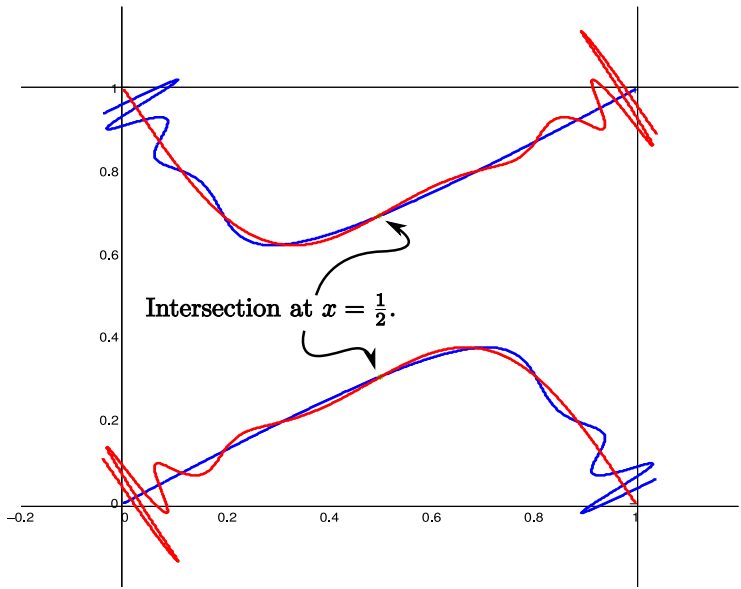}
  \caption{Intersection of the stable and unstable manifolds at $x =
    \frac{1}{2}$.}
\label{fig:manifolds}
\end{center}
\end{figure}

It follows from a simple symmetry argument that for all $\varepsilon
\neq 0$, the stable and unstable manifolds of the hyperbolic fixed
point intersect at $x=\frac{1}{2}$ (see
Figure~\ref{fig:manifolds}). Denote by $(\frac{1}{2},y_\varepsilon)$
such a homoclinic point. It is rather easy to approximate this point
numerically as one needs only to follow an approximation of either the
stable or unstable manifold until it crosses $x =
\frac{1}{2}$. Moreover, we can see from Figure~\ref{fig:manifolds}
that between the point $(\frac{1}{2},y_\varepsilon)$ and its image
$f_\varepsilon(\frac{1}{2},y_\varepsilon) = (\frac{1}{2} +
y_\varepsilon, y_\varepsilon)$ there is homoclinic point which we
denote as $(u,v)$. It is also relatively easy approximating
numerically this point based on having a good approximation of the
stable and unstable manifolds and the point
$(\frac{1}{2},y_\varepsilon)$. Using these points, one can proceed in
two different ways in order to get index pairs for and isolated
invariant set which belong to the homoclinic tangle of the fixed
point.

The first way is to find the intersection of the stable and unstable
manifolds at $x = \frac{1}{2}$ at high accuracy, and iterate this
point enough times forward and backwards to get close enough to the
fixed point. Considering all of these iterates and the fixed point as
our invariant set, we may grow an isolating neighborhood and index
pairs associated to them using Algorithm~\ref{alg:grow-ins}. This is
perhaps the quickest way to get an index pair, as we already know
where the isolated invariant set is. The downside is that in order for
this approach to be successful, we need to compute the first
homoclinic intersection with very high accuracy, as the homoclinic
connection may be lengthy, and after enough iterations the error may
become large enough to not give us a good approximation of the
invariant set.

The second approach is to compute the connections using shortest path
algorithms as in Algorithn~\ref{alg:connections}. This approach
requires less precision in the computation of the homoclinic
intersection, but is slower than the first approach. However, it is
faster than a blind shortest-path search because it takes in
consideration the dynamics of $f_\varepsilon$: recalling the action
(\ref{eqn:action}) of $f_\varepsilon$, we have $\mathcal{A}(x_1,x_2) =
h(x_1,x_2)$ and so we can compute the \emph{averaged action}
$\tilde{\mathcal{A}}:\mathcal{G}\times\mathcal{G}\rightarrow\mathbb{R}$
from box $\mathcal{B}_i$ to $\mathcal{B}_j$ as the average of
$\mathcal{A}(x,y)$ with $x\in|\mathcal{B}_i|$ and
$y\in|\mathcal{B}_j|$. Thus we can re-weigh the graph representing the
map on boxes, from having every edge of weight 1, to having the edge
going from $\mathcal{B}_i$ to $\mathcal{B}_j$ weight $K +
\tilde{\mathcal{A}}(\mathcal{B}_i,\mathcal{B}_j)$, where $K$ is any
positive number satisfying
$$K > K^* \equiv \max_{(x_1,x_2)\in\mathbb{T}^2}\left|\mathcal{A}(x_1,x_2)\right|,$$
uniform for all $\mathcal{B}_i$. Different choices of $K$ do not
necessarily give the same shortest (or cheapest) path: higher $K$
gives more weight to the number of edges in the path (which one might
use when working at a low depth), while lower $K$ gives more weight to
the action (which one might use at a high depth). Thus, searching for
shortest paths in the graph with the new weights, we have a better
chance to compute the right connection at the first try, as the
connections which are computed contain, on average, the least action
from the beginning box to the end box.

This second method turns out to be a better fit in our setting, as
numerically iterating the point $(\frac{1}{2},y_\varepsilon)$
introduces error which is multiplied upon each iteration.  Thus, the
error grows exponentially fast until the iterates reach a small
neighborhood of the fixed point, and this error is further exacerbated
by the fact that we will mostly treat $\varepsilon$ as an interval.
The second method still requires some precision, but it is modest
enough that we can efficiently obtain it using the parameterization
method~\cite{parameterization3}.  Algorithm~\ref{alg:main} summarizes
our construction of index pairs which contain the homoclinic tangle of
the hyperbolic fixed point.

As $\varepsilon$ decreases, both the area of the Birkhoff zone of
instability and the angle of intersection of stable and unstable
manifolds of the hyperbolic fixed point $(\frac{1}{2}, y_\varepsilon)$
decrease \emph{exponentially fast} with
$\varepsilon$~\cite{gelfreich}. Thus since the intersection of the
invariant manifolds is barely transversal, the index pairs associated
to the homoclinic orbits have to stretch out considerably across the
unstable manifolds in order to achieve isolation (see
Figure~\ref{fig:zoom}).  This in turn implies that
Algorithm~\ref{alg:grow-ins} takes more iterations to cover the
invariant set.  Moreover, the size of the boxes is necessarily
exponentially small with $\varepsilon$, and so the number of boxes
needed to create the index pairs increases rapidly. The bottom line is
that the complete automation of the procedure prevents us of having to
create the index pairs by hand, which for low $\varepsilon$ must be an
extremely difficult task.

KAM theory asserts that for $|\varepsilon|$ small enough (roughly
$|\varepsilon|<.971$~\cite{Jungreis}), there is a positive measure set
of homotopically non-trivial invariant circles on which the dynamics
of $f_\varepsilon$ is conjugate to irrational rotations. In this case
the invariant circles folliate the cylinder and serve as obstructions
to orbits from wandering all over the cylinder, i.e., each orbit is
confined to an area bounded by KAM circles. In this case, the
topological entropy of $f_\varepsilon$ is concentrated in the Birkhoff
zone of instability associated to the homoclinic tangle of the
hyperbolic fixed point. Once $\varepsilon>\varepsilon^*\approx .971$
there remain no homotopically non-trivial KAM circles to bound the $y$
coordinate of the orbit of a point and one has then hope to find
connecting orbits between different hyperbolic periodic orbits.

We apply our methods to obtain three types of results:
\begin{itemize}
\item For \emph{all} $f_\varepsilon$ with $\varepsilon\in[.7,2]$, we
  give a positive lower bound for its topological entropy. This is
  done by treating $\varepsilon$ as an interval. An advantage of
  having the procedure automated is that one can easily study a
  parameter-depending system at different values of the
  parameter. Treating the parameter as an interval allows us to detect
  behavior which is common to all values of the parameter in the
  interval. This is done in section~\ref{sec:intervals}.
\item Not treating $\varepsilon$ as an interval allows our method to
  go further and obtains positive bounds for the much lower value of
  $\varepsilon = \frac{1}{2}$. This is done in
  section~\ref{sec:lowest}.
\item Our examples in section~\ref{sec:periods} combine the new
  methods of this paper with the spirit of~\cite{DFT} of connecting
  periodic orbits to find better entropy bounds, which will illustrate
  for the case $\varepsilon = 2$.
\end{itemize}

We remark that in~\cite[\S 4.1]{jay1} an alternate approach for the
creation of index pairs for the standard map is given. It is done
through set-oriented methods which are based on following the
discretized dynamics along the discretized stable and unstable
manifolds. We do not know how this approach would perform when
treating $\varepsilon$ as an interval, although we suspect it would
perform equally well. Our bottom-up approach for constructing index
pairs requires the computation of fewer box images, but in general we
may make use of slightly more \emph{a-priori} information such as the
knowledge of where hyperbolic invariant sets are.  The algorithms
in~\cite{jay1} (and indeed those of \cite[\S 3.1]{DFT}) require less
\emph{a-priori} knowledge of hyperbolic, invariant sets, but require a
greater number of box-image computations (see~\cite{jay1} for more
details).

\rafnote{Edited above and fixed below.}

We should point out that the bounds we provide in the following
sections are close to some of the non-rigorous bounds given
in~\cite{tanikawa}. To our knowledge there are no other bounds in the
literature for the topological entropy of the standard map for small
values of $\varepsilon$. It is expected that the entropy is
exponentially small as $\varepsilon\rightarrow 0$~\cite{gelfreich},
while it is known that the entropy grows at least logarithmically in
$\varepsilon$ as $\epsilon\rightarrow\infty$~\cite{knill}. But for
small values of $\varepsilon$, we have not been able to find
computational bounds besides the ones already cited.

\subsection{Parameter exploration}
\label{sec:intervals}
As remarked earlier, when $\varepsilon<\varepsilon^*$ there exist
invariant KAM circles which prevent the connections between many
periodic orbits. This forces us to concentrate on the homoclinic
tangle of the hyperbolic fixed point.  The results from this section
are obtained using Algorithm~\ref{alg:main} and summarized in
Theorem~\ref{thm:main}.

\begin{algorithm}[ht]
%\begin{algorithm}[!ht]
\caption{Creating index pairs for $f_\varepsilon$ using the homoclinic orbits of the fixed point}
\begin{algorithmic}
  \STATE Input: $\bar{\varepsilon} = [\varepsilon^-,\varepsilon^+]$, $f_{\bar{\varepsilon}}$
  \STATE Let $y_{\bar{\varepsilon}} = y_{\frac{\varepsilon^-+\varepsilon^+}{2}}$ and compute $\left(\frac{1}{2},y_{\bar{\varepsilon}}\right)$ using~\cite{parameterization3}.
  \STATE $H = \left\{\left(\frac{1}{2},y_{\bar{\varepsilon}}\right),
    \left(\frac{1}{2},1 - y_{\bar{\varepsilon}}\right),
    (u,v),(1-u,1-v),(0,0)\right\}$
  \STATE $X = \bigcup_{p\neq q \in H}(p,q)$ 
  \STATE Find connections using Algorithm~\ref{alg:connections} and the weighted graph using the averaged action $\tilde{\mathcal{A}}$.
  \STATE Grow the index pairs using Algorithm~\ref{alg:grow-ins}.
\end{algorithmic}
\label{alg:main}
\end{algorithm}

By performing the computations using $\varepsilon$ as an interval, we
are proving behavior which is common for all $f_\varepsilon$ within
such interval. In such case then our guess for the homoclinic
intersection $(\frac{1}{2},y_\varepsilon)$ is done only for one point
in the interval (the midpoint).  In general it is easier to isolate an
invariant set for smaller parameter intervals, since the wider the
interval, the more general the isolation must be.  In our setting, it
is much easier to isolate homoclinic connections for $\varepsilon >
\varepsilon^*$ than it is for $\varepsilon < \varepsilon^*$.  To
reflect this, using a crude approximation of $\varepsilon^* \approx
1.0$ for ease of bookkeeping, we use intervals of size 0.005 for
$\varepsilon \geq 1.0$, but we shrink our interval size to 0.001 for
$\varepsilon < 1.0$.

As $\varepsilon$ decreases, the size of the boxes we use decreases,
and the length of the homoclinic excursion increases, leading to
another increase in the number of boxes needed and a longer running
time of Algorithm~\ref{alg:grow-ins}.  At some point, it becomes
computationaly unrealistic to continue; we stopped somewhat before
this point, when the interval computations took roughly 40 hours for
$\bar{\varepsilon} = [0.700,0.701]$.  See
section~\ref{sec:implementation} for further discussion of the
implementation and efficiency.

\begin{theorem}
\label{thm:main}
The topological entropy of the standard map $f_\varepsilon$ for
$\varepsilon \in [.7, 2]$ is bounded from below by the step function
given in Figure~\ref{fig:step-function}.  In particular, we have
$h(f_\varepsilon)>0.2$ for all $\varepsilon \in [.7, 2]$.  The precise
individual values for each subinterval are given in
Appendix~\ref{app:1}.
\begin{figure}[ht]
\begin{center}
  \includegraphics[width = 5 in]{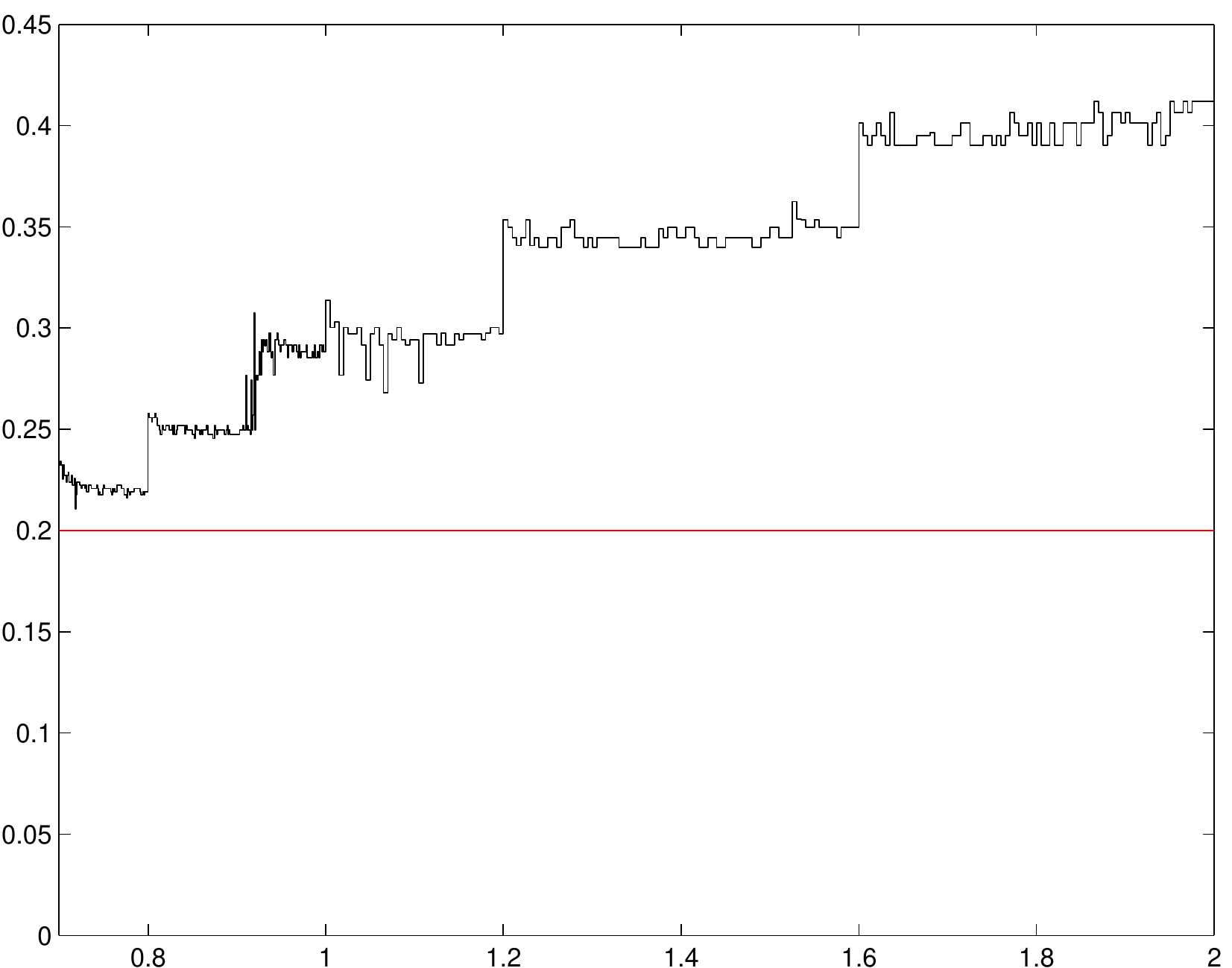}
%\vspace{-7mm}
  \caption{Lower bounds of $h(f_{\varepsilon})$ as a function of
    ${\varepsilon}$ in the interval $[.7, 2]$.  Note that all bounds
    in this interval exceed $0.2$.}
\label{fig:step-function}
\vspace{-3mm}
\end{center}
\end{figure}
\end{theorem}

\begin{proof}
  For each of the $\varepsilon$-intervals $\bar{\varepsilon} =
  [\varepsilon^-,\varepsilon^+]$ on the table found in
  Appendix~\ref{app:1}, we get an index pair for an isolated invariant
  set of $f_{\bar{\varepsilon}}$ using Algorithm~\ref{alg:main}. Using
  then Algorithm 5, 6, 7, and 8 and Theorem 3.6 from~\cite{DFT}, which
  is essentially amounts to a finite number of checks using
  Corollary~\ref{cor:Waz}, we prove a semi-conjugacy to a subshift of
  finite type, from which we get a bound on the entropy by bounding
  the spectral radius of the associated matrix.
\end{proof}

%\begin{corollary}
%For $\varepsilon \in [.7, 2]$, we have $h(f_\varepsilon) \geq 0.095 + 0.15 \varepsilon$.
%\end{corollary}

There are three aparent scales on which our lower bounds for
$h(f_\varepsilon)$ change with respect to $\varepsilon$: global,
semi-local, and local. Clearly there is an evident \emph{global}
increase of the entropy bounds as $\varepsilon$ increases, as is to
expected. On a semi-local level, there are a few intervals (roughly
$[1.6,2],[1.2,1.6]$, $[.92,1.2]$, $[.8,.92]$, $[.75,.8]$) on which the
bounds for $h(f_\varepsilon)$ seem to hover around a fixed value per
interval. This is due to using the same depth on such intervals. As
$\varepsilon$ decreases, we need to increase the depth. Locally, the
apparent irregularity of the function of lower bounds is due to the
nature of our automated approach: the accuracy of the guess
$\left(\frac{1}{2},y_{\bar{\varepsilon}}\right)$ varies per interval,
as does the computation of the averaged action, et cetera.

\subsection{Positive bound of $h(f_\varepsilon)$ for lowest $\varepsilon$.}
\label{sec:lowest}

In this section we show an example of an index pair for the standard
map for $\varepsilon = \frac{1}{2}$. This value was picked because it
is small enough that we can illustrate the strengths of our algorithms
in tight places while keeping the computation times reasonable.  Using
Algorithm~\ref{alg:main} with $\varepsilon = \frac{1}{2}$ we get an
index pair shown in Figure~\ref{fig:zoom} (although it is barely
visible).

\begin{theorem}
\label{thm:half}
The topological entropy for the standard map $f_\varepsilon$ when
$\varepsilon = \frac{1}{2}$ is bounded below by $0.1732515918346$.
\end{theorem}

The proof is the same as in Theorem~\ref{thm:main}. The tree from
which the index pair obtained for Theorem~\ref{thm:half} was obtained
contains 568,754 boxes. Among those, 281,530 are in the index
pair. Roughly a quarter of the boxes in the index pair form the exit
set (64,518). This index pair $(P_1,P_0)$ gives us an induced map
which acts on $H_1(P_1, P_0;\mathbb{Z}) = \mathbb{Z}^{1801}$ but is
reduced to a SFT in 73 symbols.
\begin{figure}[!t]
\begin{center}
\includegraphics[width = 3.5 in]{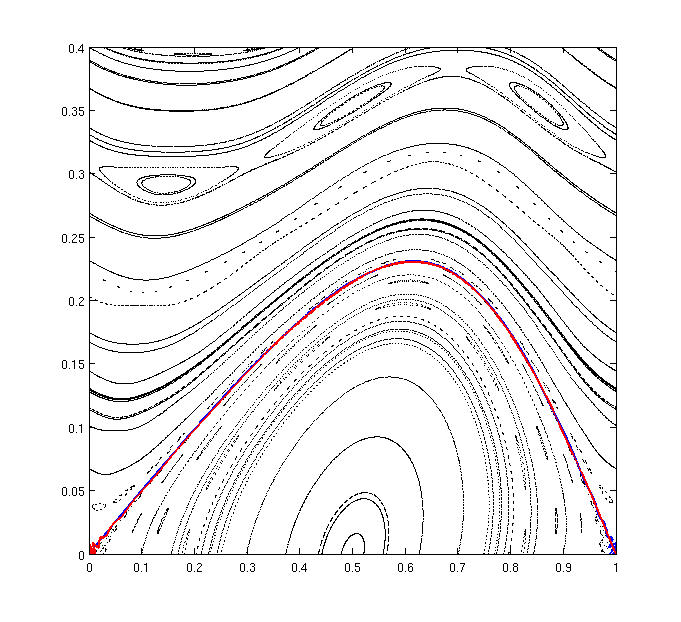}
\vspace{-5mm}
\includegraphics[width = 3.5 in]{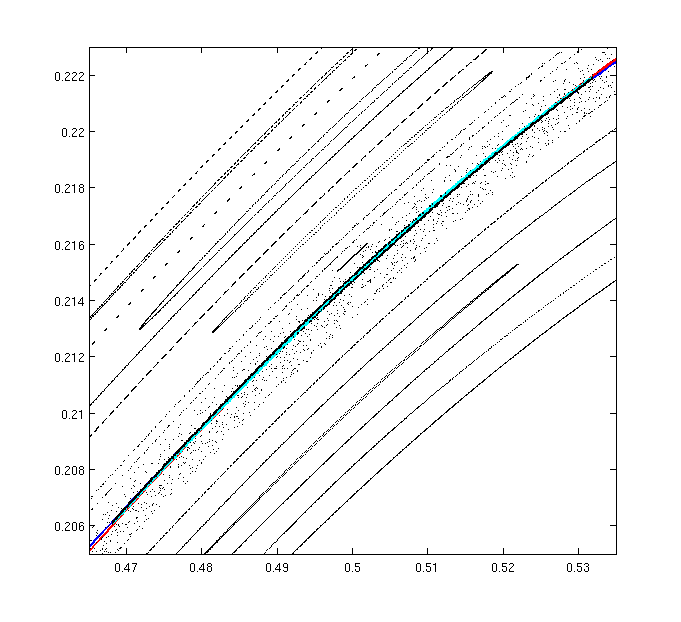}
\vspace{-3 mm}
\caption{\tiny{On top, a plot of the trajectories of the standard map
  for $\varepsilon = \frac{1}{2}$ along with the stable and unstable
  manifolds of the hyperbolic fixed point which seem to overlap. On
  the bottom, a close-up of a component of the index pair yielding
  Theorem~\ref{thm:half} which contains the homoclinic point at $x =
  \frac{1}{2}$ along with trajectories, some of which are part of KAM
  circles squeezing the component. It overlays the stable and unstable
  manifolds, whose angle of intersection is very small, causing the
  index pair to be very sheared.}}
\label{fig:zoom}
\vspace{-2.5mm}
\end{center}
\end{figure}

Figure~\ref{fig:zoom} shows, on top, the plot of some trajectories for
the standard map at $\varepsilon = \frac{1}{2}$, of the stable and
unstable manifolds, and an index pair for the homoclinic orbits. On
the bottom is a close-up to a component of the index pair squeezed by
KAM tori and a barely transversal intersection of the stable and
unstable manifolds. The boxes making up this index pair are of sides
of size $2^{-15}$. The strength of our ``growing-out'' approach is
that the creation of such index pairs in tight places is achievable
and can be automated.

\subsection{Higher periods}
\label{sec:periods}

When $\varepsilon > \varepsilon^*$ all homotopically non-trivial KAM
circles are vanished, thus it is possible to connect different
periodic orbits. The appendix of~\cite{Greene} contains an algorithm
for finding periodic orbits for the standard map. We implement this
method to find the periodic orbits which we use to grow index pairs.

\begin{algorithm}[ht]
%\begin{algorithm}[!ht]
\caption{Creating index pairs using the homoclinic orbits of the fixed point and periodic orbits}
\begin{algorithmic}
  \STATE Input: $\bar{\varepsilon} = [\varepsilon^-,\varepsilon^+]$, $f_{\bar{\varepsilon}}$, $P\in\mathbb{N}$
  \STATE Let $y_{\bar{\varepsilon}} = y_{\frac{\varepsilon^-+\varepsilon^+}{2}}$ and compute $\left(\frac{1}{2},y_{\bar{\varepsilon}}\right)$ using~\cite{parameterization3}.
  \STATE $H_1 = \left\{(0,0), \left(\frac{1}{2},y_{\bar{\varepsilon}}\right), \left(\frac{1}{2},1 - y_{\bar{\varepsilon}}\right), (u,v), (1-u,1-v)\right\}$
  \STATE $H_2 =$ hyperbolic periodic orbits of $f_{\bar{\varepsilon}}$ up to period $P$ (computed using the appendix in~\cite{Greene})
  \STATE $X = \bigcup_{p\neq q\in (H_1 \cup H_2)}(p,q)$ 
  \STATE Find connections using Algorithm~\ref{alg:connections} and the weighted graph using averaged action $\tilde{\mathcal{A}}$.
  \STATE Grow the index pairs using Algorithm~\ref{alg:grow-ins}.
\end{algorithmic}
\label{alg:periodic}
\end{algorithm}

Algorithm~\ref{alg:periodic} is essentially the main strategy employed
in~\cite{DFT}. In that paper, good index pairs were found by finding
pairwise connections between periodic orbits. Such an approach can be
slightly generalized by looking for pairwise connections between
hyperbolic, invariant sets, which is what Algorithm~\ref{alg:periodic}
does.

We apply the algorithm to $\varepsilon = 2.0$, with maximum period
$P=2$. Figure~\ref{fig:connections2} shows the index pair for this
computation. We remark that besides finding pairwise connections
between hyperbolic periodic orbits, we find connections between
periodic orbits and the homoclinic orbit $(\frac{1}{2},y_\varepsilon)$
and $(u,v)$ mentioned in section~\ref{sec:intervals}. This allows
us to find richer dynamics and to achieve higher entropy bounds. The
result from this index pair is summarized in the following theorem.

\begin{theorem}
\label{thm:connections2}
The topological entropy for the standard map $f_\varepsilon$ for
$\varepsilon = 2$ is bounded below by $0.44722970117798$.
\end{theorem}

The proof is again similar to Theorem~\ref{thm:main}.  The index pair
$(P_1,P_0)$ has a total of 8600 boxes, and gives us an induced map
which acts on $H_1(P_1, P_0;\mathbb{Z}) = \mathbb{Z}^{105}$ which is
reduced to a SFT in 59 symbols.

\begin{figure}[!ht]
\begin{center}
  \includegraphics[width = 4.2 in]{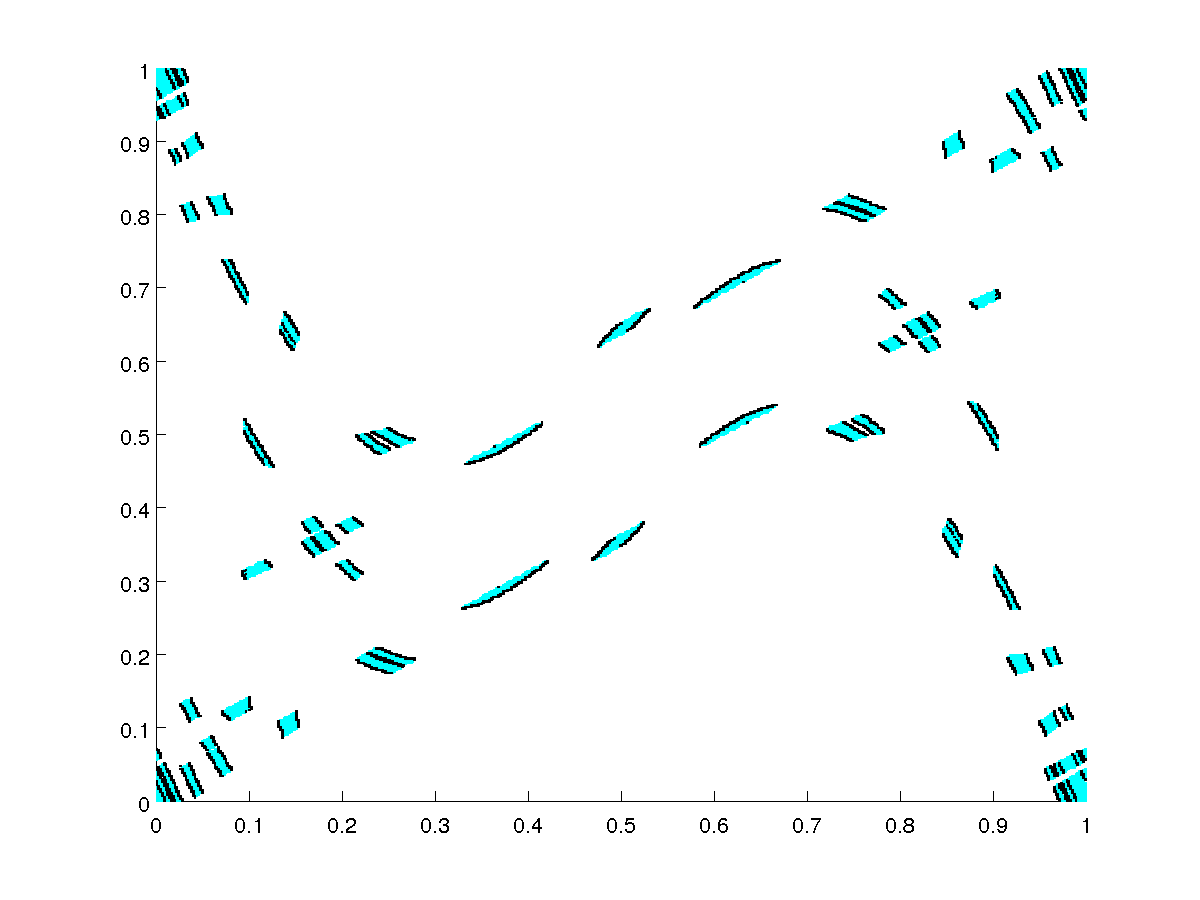}
\vspace{-8mm}
\caption{Index pair obtained through Algorithm~\ref{alg:periodic} for $\varepsilon = 2$ and $P = 2$ at depth $9$.}
\label{fig:connections2}
\vspace{-3mm}
\end{center}
\end{figure}

\begin{figure}[!ht]
\begin{center}
  \includegraphics[width = 4.2 in]{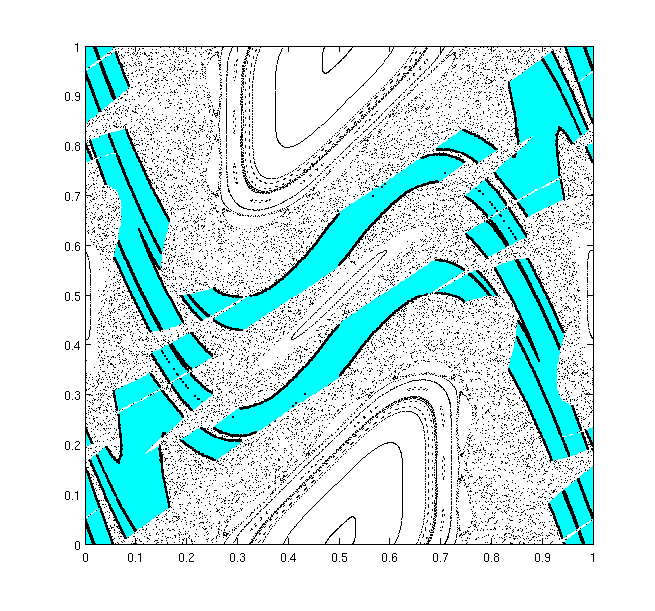}
\vspace{-8mm}
\caption{Index pair obtained by adding periodic orbits of period 3 and 4 to the example in Figure~\ref{fig:connections2}}
\label{fig:awesome}
\vspace{-3mm}
\end{center}
\end{figure}

This result shows that our bounds in Theorem~\ref{thm:main} are
probably suboptimal, since our interval bound for $[1.995,2.000]$ is
lower than in Theorem~\ref{thm:connections2}. It is not surprising
that by adding connections to another hyperbolic periodic orbit we may
find a higher entropy bound.

Algorithm~\ref{alg:periodic} can be quite effective when $P$ is small,
but as $P$ grows, there is large increase (quadratic at the very
least, often exponential) in the number of connections to compute,
especially since the number of period-$P$ orbits grows rapidly with
$P$ for chaotic maps.  Moreover, the number of long connections
increases, which as discussed in section~\ref{sec:approx} makes
Algorithm~\ref{alg:connections} work even harder.  As suggested in
that section, we instead turn to computing periodic orbits of higher
period rather than computing connections explicitly.

For our last example, we apply this approach on top of the previous
example from Theorem \ref{thm:connections2} to produce a very strong
index pair.  That is, we run Algorithm~\ref{alg:periodic} with $P=2$
and then simply add all hyperbolic orbits of period 3 and 4 (there are
four of each) \emph{without} adding any further connections.  The
resulting index pair, shown in Figure~\ref{fig:awesome}, clearly
benefits from these ``natural'' connections.

\begin{theorem}
\label{thm:awesome}
The topological entropy for the standard map $f_\varepsilon$ for
$\varepsilon = 2$ is bounded below by $0.54518888942276$.
\end{theorem}

The index pair $(P_1,P_0)$ for Theorem~\ref{thm:awesome} has a total
of 64,185 boxes, with 5,839 in the exit set.  The induced map acts on
$H_1(P_1, P_0;\mathbb{Z}) = \mathbb{Z}^{138}$ and reduces to a SFT on
only 41 symbols.

\subsection{Notes on efficiency and implementation}
\label{sec:implementation}

The computations in this paper were performed in MATLAB on machines
with between 1 and 2 gigabytes (GB) of memory and with clock speeds
between 1.5 and 2.5 gigahertz.  Runtimes ranged from 3 or 4 minutes
for $\bar{\varepsilon}=[1.995,2]$ and $\varepsilon=2.0$, to almost 2
days for $\bar{\varepsilon}=[0.700,0.701]$ and roughly 5 days for
$\varepsilon=0.5$.  As discussed in section~\ref{sec:intervals} and
the beginning of section~\ref{sec:computations}, there was a roughly
inverse exponential relationship between $\varepsilon$ and the runtime
for the intermediate values.

The two most time-consuming subroutines for our computations were
Algorithms~\ref{alg:connections} and~\ref{alg:grow-ins}, which of
course is one reason for our focus on them in this work.  As
$\varepsilon$ decreased, however, Algorithm~\ref{alg:grow-ins}
dominated the runtime, particularly in the bookkeeping step
(maintaining the correct box numbers, as discussed in
section~\ref{sec:datastructures}) and the insertion step, when new
boxes are inserted into the tree.  While the insertions cannot be
avoided, this does suggest that a better tree implementation could
enhance performance greatly.

To conclude, we would like to reiterate how difficult it would be to
reproduce our results for low $\varepsilon$, namely
Theorem~\ref{thm:half} and the lower intervals of
Theorem~\ref{thm:main}, using index pair algorithms from~\cite{DFT}.
As mentioned in section~\ref{sec:growing}, our algorithms are
considerably more memory efficient in certain situations, which we now
explore concretely.  Consider the index pair we obtained in
Theorem~\ref{thm:half} when $\varepsilon=0.5$; there were 568754 boxes
in the index pair, at depth 15, and the adjacency and transition
matrices took up about 0.1GB of memory using our approach.  Using
the~\cite{DFT} approach, one would need to compute the map on
\emph{all} boxes at depth 15.  Using a conservative estimate of 190
bytes per box to store the image and adjacency information (the
average for our index pair was 194.38 bytes), this would require
204GB, which is beyond reasonable at the time of this writing.  More
to the point, our bottom-up approach is clearly orders of magnitude more
efficient in terms of memory.  When one considers the graph
computations that would need to be carried out on the resulting
$2^{30}$-node graph (the transition matrix), it becomes even clearer
that our computations would have been impractical using the top-down approach
of~\cite{DFT}.

%45 minutes for the $\bar{\varepsilon}=[0.900,0.901]$ case

\appendix 
\section{Precise bounds for $h(f_\varepsilon)$}
Below we list the actual values for the lower bounds in Theorem
\ref{thm:main} the last column indicates the number of symbols for
each SFT whose entropy bounds that of $f_\varepsilon$ for
$\varepsilon$ in each interval.

\label{app:1}
% Below is the table of the lower bounds for all intervals found in
% Figure~\ref{fig:step-function}. The ``sym'' column represents the
% size of the alphabet for the subshift of finite type whose
% topological entropy bounds $h(f_\varepsilon)$.  \newpage
\input{table.tex}

\bibliographystyle{amsalpha}
\bibliography{biblio}{}
\end{document}

%% file: table.tex
{\scriptsize
\begin{center}
\begin{tabular}{cc|c}
\hspace{-2pt}$\varepsilon$ interval & $h(f_\varepsilon)\geq$ & \hspace{-2pt}sym\hspace{-2pt} \\[1pt] \hline \\[-6pt] 
\hspace{-4pt}[0.700, 0.701]\hspace{-2pt} & 0.232286 & 51 \\ 
\hspace{-4pt}[0.701, 0.702]\hspace{-2pt} & 0.234189 & 51 \\ 
\hspace{-4pt}[0.702, 0.703]\hspace{-2pt} & 0.232286 & 51 \\ 
\hspace{-4pt}[0.703, 0.704]\hspace{-2pt} & 0.232286 & 51 \\ 
\hspace{-4pt}[0.704, 0.705]\hspace{-2pt} & 0.225504 & 53 \\ 
\hspace{-4pt}[0.705, 0.706]\hspace{-2pt} & 0.232286 & 51 \\ 
\hspace{-4pt}[0.706, 0.707]\hspace{-2pt} & 0.227178 & 51 \\ 
\hspace{-4pt}[0.707, 0.708]\hspace{-2pt} & 0.227178 & 51 \\ 
\hspace{-4pt}[0.708, 0.709]\hspace{-2pt} & 0.223988 & 53 \\ 
\hspace{-4pt}[0.709, 0.710]\hspace{-2pt} & 0.227178 & 53 \\ 
\hspace{-4pt}[0.710, 0.711]\hspace{-2pt} & 0.228742 & 53 \\ 
\hspace{-4pt}[0.711, 0.712]\hspace{-2pt} & 0.223988 & 53 \\ 
\hspace{-4pt}[0.712, 0.713]\hspace{-2pt} & 0.223988 & 53 \\ 
\hspace{-4pt}[0.713, 0.714]\hspace{-2pt} & 0.223988 & 53 \\ 
\hspace{-4pt}[0.714, 0.715]\hspace{-2pt} & 0.227178 & 53 \\ 
\hspace{-4pt}[0.715, 0.716]\hspace{-2pt} & 0.222542 & 53 \\ 
\hspace{-4pt}[0.716, 0.717]\hspace{-2pt} & 0.222365 & 53 \\ 
\hspace{-4pt}[0.717, 0.718]\hspace{-2pt} & 0.225679 & 53 \\ 
\hspace{-4pt}[0.718, 0.719]\hspace{-2pt} & 0.210614 & 39 \\ 
\hspace{-4pt}[0.719, 0.720]\hspace{-2pt} & 0.217650 & 55 \\ 
\hspace{-4pt}[0.720, 0.721]\hspace{-2pt} & 0.223988 & 53 \\ 
\hspace{-4pt}[0.721, 0.722]\hspace{-2pt} & 0.223988 & 53 \\ 
\hspace{-4pt}[0.722, 0.723]\hspace{-2pt} & 0.223988 & 53 \\ 
\hspace{-4pt}[0.723, 0.724]\hspace{-2pt} & 0.222542 & 53 \\ 
\hspace{-4pt}[0.724, 0.725]\hspace{-2pt} & 0.222542 & 53 \\ 
\hspace{-4pt}[0.725, 0.726]\hspace{-2pt} & 0.220813 & 55 \\ 
\hspace{-4pt}[0.726, 0.727]\hspace{-2pt} & 0.222542 & 53 \\ 
\hspace{-4pt}[0.727, 0.728]\hspace{-2pt} & 0.222542 & 53 \\ 
\hspace{-4pt}[0.728, 0.729]\hspace{-2pt} & 0.222542 & 55 \\ 
\hspace{-4pt}[0.729, 0.730]\hspace{-2pt} & 0.220813 & 55 \\ 
\hspace{-4pt}[0.730, 0.731]\hspace{-2pt} & 0.222365 & 55 \\ 
\hspace{-4pt}[0.731, 0.732]\hspace{-2pt} & 0.219153 & 55 \\ 
\hspace{-4pt}[0.732, 0.733]\hspace{-2pt} & 0.219153 & 55 \\ 
\hspace{-4pt}[0.733, 0.734]\hspace{-2pt} & 0.222542 & 55 \\ 
\hspace{-4pt}[0.734, 0.735]\hspace{-2pt} & 0.222542 & 55 \\ 
\hspace{-4pt}[0.735, 0.736]\hspace{-2pt} & 0.222542 & 55 \\ 
\hspace{-4pt}[0.736, 0.737]\hspace{-2pt} & 0.220813 & 55 \\ 
\hspace{-4pt}[0.737, 0.738]\hspace{-2pt} & 0.220813 & 55 \\ 
\hspace{-4pt}[0.738, 0.739]\hspace{-2pt} & 0.220813 & 55 \\ 
\hspace{-4pt}[0.739, 0.740]\hspace{-2pt} & 0.220813 & 55 \\ 
\hspace{-4pt}[0.740, 0.741]\hspace{-2pt} & 0.220813 & 55 \\ 
\hspace{-4pt}[0.741, 0.742]\hspace{-2pt} & 0.220813 & 55 \\ 
\hspace{-4pt}[0.742, 0.743]\hspace{-2pt} & 0.222542 & 55 \\ 
\hspace{-4pt}[0.743, 0.744]\hspace{-2pt} & 0.220813 & 55 \\ 
\hspace{-4pt}[0.744, 0.745]\hspace{-2pt} & 0.217650 & 55 \\ 
\hspace{-4pt}[0.745, 0.746]\hspace{-2pt} & 0.219153 & 55 \\ 
\hspace{-4pt}[0.746, 0.747]\hspace{-2pt} & 0.217650 & 55 \\ 
\hspace{-4pt}[0.747, 0.748]\hspace{-2pt} & 0.217650 & 55 \\ 
\hspace{-4pt}[0.748, 0.749]\hspace{-2pt} & 0.217650 & 55 \\ 
\hspace{-4pt}[0.749, 0.750]\hspace{-2pt} & 0.220813 & 55 \\ 
\hspace{-4pt}[0.750, 0.751]\hspace{-2pt} & 0.222542 & 55 \\ 
\hspace{-4pt}[0.751, 0.752]\hspace{-2pt} & 0.220813 & 55 \\ 
\hspace{-4pt}[0.752, 0.753]\hspace{-2pt} & 0.220813 & 55 \\ 
\hspace{-4pt}[0.753, 0.754]\hspace{-2pt} & 0.220813 & 55 \\ 
\hspace{-4pt}[0.754, 0.755]\hspace{-2pt} & 0.220813 & 55 \\ 
\hspace{-4pt}[0.755, 0.756]\hspace{-2pt} & 0.220813 & 55 \\ 
\end{tabular}
\hspace{0.2 cm}
\begin{tabular}{cc|c}
\hspace{-2pt}$\varepsilon$ interval & $h(f_\varepsilon)\geq$ & \hspace{-2pt}sym\hspace{-2pt} \\[1pt] \hline \\[-6pt] 
\hspace{-4pt}[0.756, 0.757]\hspace{-2pt} & 0.220813 & 55 \\ 
\hspace{-4pt}[0.757, 0.758]\hspace{-2pt} & 0.220813 & 55 \\ 
\hspace{-4pt}[0.758, 0.759]\hspace{-2pt} & 0.219153 & 55 \\ 
\hspace{-4pt}[0.759, 0.760]\hspace{-2pt} & 0.217650 & 55 \\ 
\hspace{-4pt}[0.760, 0.761]\hspace{-2pt} & 0.220813 & 55 \\ 
\hspace{-4pt}[0.761, 0.762]\hspace{-2pt} & 0.219153 & 55 \\ 
\hspace{-4pt}[0.762, 0.763]\hspace{-2pt} & 0.220813 & 55 \\ 
\hspace{-4pt}[0.763, 0.764]\hspace{-2pt} & 0.219153 & 55 \\ 
\hspace{-4pt}[0.764, 0.765]\hspace{-2pt} & 0.219153 & 55 \\ 
\hspace{-4pt}[0.765, 0.766]\hspace{-2pt} & 0.222542 & 55 \\ 
\hspace{-4pt}[0.766, 0.767]\hspace{-2pt} & 0.222542 & 55 \\ 
\hspace{-4pt}[0.767, 0.768]\hspace{-2pt} & 0.222542 & 55 \\ 
\hspace{-4pt}[0.768, 0.769]\hspace{-2pt} & 0.222542 & 55 \\ 
\hspace{-4pt}[0.769, 0.770]\hspace{-2pt} & 0.222542 & 55 \\ 
\hspace{-4pt}[0.770, 0.771]\hspace{-2pt} & 0.220813 & 55 \\ 
\hspace{-4pt}[0.771, 0.772]\hspace{-2pt} & 0.220813 & 55 \\ 
\hspace{-4pt}[0.772, 0.773]\hspace{-2pt} & 0.220813 & 55 \\ 
\hspace{-4pt}[0.773, 0.774]\hspace{-2pt} & 0.217650 & 55 \\ 
\hspace{-4pt}[0.774, 0.775]\hspace{-2pt} & 0.217650 & 55 \\ 
\hspace{-4pt}[0.775, 0.776]\hspace{-2pt} & 0.217650 & 55 \\ 
\hspace{-4pt}[0.776, 0.777]\hspace{-2pt} & 0.216045 & 55 \\ 
\hspace{-4pt}[0.777, 0.778]\hspace{-2pt} & 0.220813 & 55 \\ 
\hspace{-4pt}[0.778, 0.779]\hspace{-2pt} & 0.219153 & 55 \\ 
\hspace{-4pt}[0.779, 0.780]\hspace{-2pt} & 0.217650 & 55 \\ 
\hspace{-4pt}[0.780, 0.781]\hspace{-2pt} & 0.219153 & 55 \\ 
\hspace{-4pt}[0.781, 0.782]\hspace{-2pt} & 0.219153 & 55 \\ 
\hspace{-4pt}[0.782, 0.783]\hspace{-2pt} & 0.219153 & 55 \\ 
\hspace{-4pt}[0.783, 0.784]\hspace{-2pt} & 0.219153 & 55 \\ 
\hspace{-4pt}[0.784, 0.785]\hspace{-2pt} & 0.220813 & 55 \\ 
\hspace{-4pt}[0.785, 0.786]\hspace{-2pt} & 0.220813 & 55 \\ 
\hspace{-4pt}[0.786, 0.787]\hspace{-2pt} & 0.220813 & 55 \\ 
\hspace{-4pt}[0.787, 0.788]\hspace{-2pt} & 0.220813 & 55 \\ 
\hspace{-4pt}[0.788, 0.789]\hspace{-2pt} & 0.220813 & 55 \\ 
\hspace{-4pt}[0.789, 0.790]\hspace{-2pt} & 0.220813 & 55 \\ 
\hspace{-4pt}[0.790, 0.791]\hspace{-2pt} & 0.220813 & 55 \\ 
\hspace{-4pt}[0.791, 0.792]\hspace{-2pt} & 0.219153 & 55 \\ 
\hspace{-4pt}[0.792, 0.793]\hspace{-2pt} & 0.217650 & 55 \\ 
\hspace{-4pt}[0.793, 0.794]\hspace{-2pt} & 0.217650 & 55 \\ 
\hspace{-4pt}[0.794, 0.795]\hspace{-2pt} & 0.219153 & 55 \\ 
\hspace{-4pt}[0.795, 0.796]\hspace{-2pt} & 0.217650 & 55 \\ 
\hspace{-4pt}[0.796, 0.797]\hspace{-2pt} & 0.219153 & 55 \\ 
\hspace{-4pt}[0.797, 0.798]\hspace{-2pt} & 0.219153 & 55 \\ 
\hspace{-4pt}[0.798, 0.799]\hspace{-2pt} & 0.219153 & 55 \\ 
\hspace{-4pt}[0.799, 0.800]\hspace{-2pt} & 0.219153 & 55 \\ 
\hspace{-4pt}[0.800, 0.801]\hspace{-2pt} & 0.257972 & 43 \\ 
\hspace{-4pt}[0.801, 0.802]\hspace{-2pt} & 0.255740 & 45 \\ 
\hspace{-4pt}[0.802, 0.803]\hspace{-2pt} & 0.255740 & 43 \\ 
\hspace{-4pt}[0.803, 0.804]\hspace{-2pt} & 0.255740 & 43 \\ 
\hspace{-4pt}[0.804, 0.805]\hspace{-2pt} & 0.253746 & 45 \\ 
\hspace{-4pt}[0.805, 0.806]\hspace{-2pt} & 0.255740 & 45 \\ 
\hspace{-4pt}[0.806, 0.807]\hspace{-2pt} & 0.255740 & 45 \\ 
\hspace{-4pt}[0.807, 0.808]\hspace{-2pt} & 0.255740 & 45 \\ 
\hspace{-4pt}[0.808, 0.809]\hspace{-2pt} & 0.257972 & 43 \\ 
\hspace{-4pt}[0.809, 0.810]\hspace{-2pt} & 0.255740 & 45 \\ 
\hspace{-4pt}[0.810, 0.811]\hspace{-2pt} & 0.251858 & 45 \\ 
\hspace{-4pt}[0.811, 0.812]\hspace{-2pt} & 0.251858 & 45 \\ 
\end{tabular}
\hspace{0.2 cm}
\begin{tabular}{cc|c}
\hspace{-2pt}$\varepsilon$ interval & $h(f_\varepsilon)\geq$ & \hspace{-2pt}sym\hspace{-2pt} \\[1pt] \hline \\[-6pt] 
\hspace{-4pt}[0.812, 0.813]\hspace{-2pt} & 0.251858 & 45 \\ 
\hspace{-4pt}[0.813, 0.814]\hspace{-2pt} & 0.249549 & 45 \\ 
\hspace{-4pt}[0.814, 0.815]\hspace{-2pt} & 0.247344 & 47 \\ 
\hspace{-4pt}[0.815, 0.816]\hspace{-2pt} & 0.247344 & 47 \\ 
\hspace{-4pt}[0.816, 0.817]\hspace{-2pt} & 0.251858 & 47 \\ 
\hspace{-4pt}[0.817, 0.818]\hspace{-2pt} & 0.249549 & 45 \\ 
\hspace{-4pt}[0.818, 0.819]\hspace{-2pt} & 0.249549 & 47 \\ 
\hspace{-4pt}[0.819, 0.820]\hspace{-2pt} & 0.249549 & 47 \\ 
\hspace{-4pt}[0.820, 0.821]\hspace{-2pt} & 0.251858 & 45 \\ 
\hspace{-4pt}[0.821, 0.822]\hspace{-2pt} & 0.251858 & 47 \\ 
\hspace{-4pt}[0.822, 0.823]\hspace{-2pt} & 0.251858 & 45 \\ 
\hspace{-4pt}[0.823, 0.824]\hspace{-2pt} & 0.251858 & 47 \\ 
\hspace{-4pt}[0.824, 0.825]\hspace{-2pt} & 0.249549 & 47 \\ 
\hspace{-4pt}[0.825, 0.826]\hspace{-2pt} & 0.249549 & 47 \\ 
\hspace{-4pt}[0.826, 0.827]\hspace{-2pt} & 0.249549 & 47 \\ 
\hspace{-4pt}[0.827, 0.828]\hspace{-2pt} & 0.251858 & 47 \\ 
\hspace{-4pt}[0.828, 0.829]\hspace{-2pt} & 0.247344 & 47 \\ 
\hspace{-4pt}[0.829, 0.830]\hspace{-2pt} & 0.251858 & 45 \\ 
\hspace{-4pt}[0.830, 0.831]\hspace{-2pt} & 0.247344 & 47 \\ 
\hspace{-4pt}[0.831, 0.832]\hspace{-2pt} & 0.247344 & 47 \\ 
\hspace{-4pt}[0.832, 0.833]\hspace{-2pt} & 0.249549 & 47 \\ 
\hspace{-4pt}[0.833, 0.834]\hspace{-2pt} & 0.251858 & 47 \\ 
\hspace{-4pt}[0.834, 0.835]\hspace{-2pt} & 0.251858 & 47 \\ 
\hspace{-4pt}[0.835, 0.836]\hspace{-2pt} & 0.251858 & 47 \\ 
\hspace{-4pt}[0.836, 0.837]\hspace{-2pt} & 0.251858 & 47 \\ 
\hspace{-4pt}[0.837, 0.838]\hspace{-2pt} & 0.251858 & 47 \\ 
\hspace{-4pt}[0.838, 0.839]\hspace{-2pt} & 0.251858 & 47 \\ 
\hspace{-4pt}[0.839, 0.840]\hspace{-2pt} & 0.251858 & 47 \\ 
\hspace{-4pt}[0.840, 0.841]\hspace{-2pt} & 0.251858 & 47 \\ 
\hspace{-4pt}[0.841, 0.842]\hspace{-2pt} & 0.247610 & 47 \\ 
\hspace{-4pt}[0.842, 0.843]\hspace{-2pt} & 0.251858 & 47 \\ 
\hspace{-4pt}[0.843, 0.844]\hspace{-2pt} & 0.251858 & 47 \\ 
\hspace{-4pt}[0.844, 0.845]\hspace{-2pt} & 0.249549 & 47 \\ 
\hspace{-4pt}[0.845, 0.846]\hspace{-2pt} & 0.249549 & 47 \\ 
\hspace{-4pt}[0.846, 0.847]\hspace{-2pt} & 0.249549 & 47 \\ 
\hspace{-4pt}[0.847, 0.848]\hspace{-2pt} & 0.249549 & 47 \\ 
\hspace{-4pt}[0.848, 0.849]\hspace{-2pt} & 0.249549 & 47 \\ 
\hspace{-4pt}[0.849, 0.850]\hspace{-2pt} & 0.249549 & 47 \\ 
\hspace{-4pt}[0.850, 0.851]\hspace{-2pt} & 0.247344 & 47 \\ 
\hspace{-4pt}[0.851, 0.852]\hspace{-2pt} & 0.247344 & 47 \\ 
\hspace{-4pt}[0.852, 0.853]\hspace{-2pt} & 0.245376 & 47 \\ 
\hspace{-4pt}[0.853, 0.854]\hspace{-2pt} & 0.251858 & 47 \\ 
\hspace{-4pt}[0.854, 0.855]\hspace{-2pt} & 0.251858 & 47 \\ 
\hspace{-4pt}[0.855, 0.856]\hspace{-2pt} & 0.249549 & 47 \\ 
\hspace{-4pt}[0.856, 0.857]\hspace{-2pt} & 0.249549 & 47 \\ 
\hspace{-4pt}[0.857, 0.858]\hspace{-2pt} & 0.249549 & 47 \\ 
\hspace{-4pt}[0.858, 0.859]\hspace{-2pt} & 0.247344 & 47 \\ 
\hspace{-4pt}[0.859, 0.860]\hspace{-2pt} & 0.247344 & 47 \\ 
\hspace{-4pt}[0.860, 0.861]\hspace{-2pt} & 0.249549 & 47 \\ 
\hspace{-4pt}[0.861, 0.862]\hspace{-2pt} & 0.247344 & 47 \\ 
\hspace{-4pt}[0.862, 0.863]\hspace{-2pt} & 0.249549 & 47 \\ 
\hspace{-4pt}[0.863, 0.864]\hspace{-2pt} & 0.249549 & 47 \\ 
\hspace{-4pt}[0.864, 0.865]\hspace{-2pt} & 0.249549 & 47 \\ 
\hspace{-4pt}[0.865, 0.866]\hspace{-2pt} & 0.249549 & 47 \\ 
\hspace{-4pt}[0.866, 0.867]\hspace{-2pt} & 0.251858 & 47 \\ 
\hspace{-4pt}[0.867, 0.868]\hspace{-2pt} & 0.247344 & 47 \\ 
\end{tabular}
\end{center}
}
{\scriptsize
\begin{center}
\begin{tabular}{cc|c}
\hspace{-2pt}$\varepsilon$ interval & $h(f_\varepsilon)\geq$ & \hspace{-2pt}sym\hspace{-2pt} \\[1pt] \hline \\[-6pt] 
\hspace{-4pt}[0.868, 0.869]\hspace{-2pt} & 0.247344 & 47 \\ 
\hspace{-4pt}[0.869, 0.870]\hspace{-2pt} & 0.247344 & 47 \\ 
\hspace{-4pt}[0.870, 0.871]\hspace{-2pt} & 0.247344 & 47 \\ 
\hspace{-4pt}[0.871, 0.872]\hspace{-2pt} & 0.247344 & 47 \\ 
\hspace{-4pt}[0.872, 0.873]\hspace{-2pt} & 0.247344 & 47 \\ 
\hspace{-4pt}[0.873, 0.874]\hspace{-2pt} & 0.245376 & 47 \\ 
\hspace{-4pt}[0.874, 0.875]\hspace{-2pt} & 0.245376 & 47 \\ 
\hspace{-4pt}[0.875, 0.876]\hspace{-2pt} & 0.251858 & 47 \\ 
\hspace{-4pt}[0.876, 0.877]\hspace{-2pt} & 0.249549 & 47 \\ 
\hspace{-4pt}[0.877, 0.878]\hspace{-2pt} & 0.247344 & 47 \\ 
\hspace{-4pt}[0.878, 0.879]\hspace{-2pt} & 0.249549 & 47 \\ 
\hspace{-4pt}[0.879, 0.880]\hspace{-2pt} & 0.249549 & 47 \\ 
\hspace{-4pt}[0.880, 0.881]\hspace{-2pt} & 0.249549 & 47 \\ 
\hspace{-4pt}[0.881, 0.882]\hspace{-2pt} & 0.249549 & 47 \\ 
\hspace{-4pt}[0.882, 0.883]\hspace{-2pt} & 0.249549 & 47 \\ 
\hspace{-4pt}[0.883, 0.884]\hspace{-2pt} & 0.249549 & 47 \\ 
\hspace{-4pt}[0.884, 0.885]\hspace{-2pt} & 0.249549 & 47 \\ 
\hspace{-4pt}[0.885, 0.886]\hspace{-2pt} & 0.249549 & 47 \\ 
\hspace{-4pt}[0.886, 0.887]\hspace{-2pt} & 0.247344 & 47 \\ 
\hspace{-4pt}[0.887, 0.888]\hspace{-2pt} & 0.247344 & 47 \\ 
\hspace{-4pt}[0.888, 0.889]\hspace{-2pt} & 0.249549 & 47 \\ 
\hspace{-4pt}[0.889, 0.890]\hspace{-2pt} & 0.251858 & 47 \\ 
\hspace{-4pt}[0.890, 0.891]\hspace{-2pt} & 0.247344 & 47 \\ 
\hspace{-4pt}[0.891, 0.892]\hspace{-2pt} & 0.249549 & 47 \\ 
\hspace{-4pt}[0.892, 0.893]\hspace{-2pt} & 0.247344 & 47 \\ 
\hspace{-4pt}[0.893, 0.894]\hspace{-2pt} & 0.247344 & 47 \\ 
\hspace{-4pt}[0.894, 0.895]\hspace{-2pt} & 0.247344 & 47 \\ 
\hspace{-4pt}[0.895, 0.896]\hspace{-2pt} & 0.247344 & 47 \\ 
\hspace{-4pt}[0.896, 0.897]\hspace{-2pt} & 0.247344 & 47 \\ 
\hspace{-4pt}[0.897, 0.898]\hspace{-2pt} & 0.247344 & 47 \\ 
\hspace{-4pt}[0.898, 0.899]\hspace{-2pt} & 0.247344 & 47 \\ 
\hspace{-4pt}[0.899, 0.900]\hspace{-2pt} & 0.247344 & 47 \\ 
\hspace{-4pt}[0.900, 0.901]\hspace{-2pt} & 0.247344 & 47 \\ 
\hspace{-4pt}[0.901, 0.902]\hspace{-2pt} & 0.247344 & 47 \\ 
\hspace{-4pt}[0.902, 0.903]\hspace{-2pt} & 0.247344 & 47 \\ 
\hspace{-4pt}[0.903, 0.904]\hspace{-2pt} & 0.249549 & 47 \\ 
\hspace{-4pt}[0.904, 0.905]\hspace{-2pt} & 0.249549 & 47 \\ 
\hspace{-4pt}[0.905, 0.906]\hspace{-2pt} & 0.249549 & 47 \\ 
\hspace{-4pt}[0.906, 0.907]\hspace{-2pt} & 0.249549 & 47 \\ 
\hspace{-4pt}[0.907, 0.908]\hspace{-2pt} & 0.251858 & 47 \\ 
\hspace{-4pt}[0.908, 0.909]\hspace{-2pt} & 0.249549 & 47 \\ 
\hspace{-4pt}[0.909, 0.910]\hspace{-2pt} & 0.249549 & 47 \\ 
\hspace{-4pt}[0.910, 0.911]\hspace{-2pt} & 0.276723 & 27 \\ 
\hspace{-4pt}[0.911, 0.912]\hspace{-2pt} & 0.249549 & 47 \\ 
\hspace{-4pt}[0.912, 0.913]\hspace{-2pt} & 0.251858 & 47 \\ 
\hspace{-4pt}[0.913, 0.914]\hspace{-2pt} & 0.249549 & 47 \\ 
\hspace{-4pt}[0.914, 0.915]\hspace{-2pt} & 0.249549 & 47 \\ 
\hspace{-4pt}[0.915, 0.916]\hspace{-2pt} & 0.247344 & 47 \\ 
\hspace{-4pt}[0.916, 0.917]\hspace{-2pt} & 0.274243 & 27 \\ 
\hspace{-4pt}[0.917, 0.918]\hspace{-2pt} & 0.249549 & 47 \\ 
\hspace{-4pt}[0.918, 0.919]\hspace{-2pt} & 0.257110 & 25 \\ 
\hspace{-4pt}[0.919, 0.920]\hspace{-2pt} & 0.307453 & 35 \\ 
\hspace{-4pt}[0.920, 0.921]\hspace{-2pt} & 0.249549 & 47 \\ 
\hspace{-4pt}[0.921, 0.922]\hspace{-2pt} & 0.276723 & 27 \\ 
\hspace{-4pt}[0.922, 0.923]\hspace{-2pt} & 0.276723 & 27 \\ 
\hspace{-4pt}[0.923, 0.924]\hspace{-2pt} & 0.274243 & 27 \\ 
\end{tabular}
\hspace{0.2 cm}
\begin{tabular}{cc|c}
\hspace{-2pt}$\varepsilon$ interval & $h(f_\varepsilon)\geq$ & \hspace{-2pt}sym\hspace{-2pt} \\[1pt] \hline \\[-6pt] 
\hspace{-4pt}[0.924, 0.925]\hspace{-2pt} & 0.276723 & 27 \\ 
\hspace{-4pt}[0.925, 0.926]\hspace{-2pt} & 0.288442 & 37 \\ 
\hspace{-4pt}[0.926, 0.927]\hspace{-2pt} & 0.276723 & 27 \\ 
\hspace{-4pt}[0.927, 0.928]\hspace{-2pt} & 0.276723 & 27 \\ 
\hspace{-4pt}[0.928, 0.929]\hspace{-2pt} & 0.294293 & 37 \\ 
\hspace{-4pt}[0.929, 0.930]\hspace{-2pt} & 0.288442 & 37 \\ 
\hspace{-4pt}[0.930, 0.931]\hspace{-2pt} & 0.294293 & 37 \\ 
\hspace{-4pt}[0.931, 0.932]\hspace{-2pt} & 0.291280 & 37 \\ 
\hspace{-4pt}[0.932, 0.933]\hspace{-2pt} & 0.291280 & 37 \\ 
\hspace{-4pt}[0.933, 0.934]\hspace{-2pt} & 0.294293 & 37 \\ 
\hspace{-4pt}[0.934, 0.935]\hspace{-2pt} & 0.288442 & 37 \\ 
\hspace{-4pt}[0.935, 0.936]\hspace{-2pt} & 0.288442 & 39 \\ 
\hspace{-4pt}[0.936, 0.937]\hspace{-2pt} & 0.297475 & 37 \\ 
\hspace{-4pt}[0.937, 0.938]\hspace{-2pt} & 0.297475 & 37 \\ 
\hspace{-4pt}[0.938, 0.939]\hspace{-2pt} & 0.288442 & 39 \\ 
\hspace{-4pt}[0.939, 0.940]\hspace{-2pt} & 0.285349 & 39 \\ 
\hspace{-4pt}[0.940, 0.941]\hspace{-2pt} & 0.288442 & 39 \\ 
\hspace{-4pt}[0.941, 0.942]\hspace{-2pt} & 0.276723 & 27 \\ 
\hspace{-4pt}[0.942, 0.943]\hspace{-2pt} & 0.276723 & 27 \\ 
\hspace{-4pt}[0.943, 0.944]\hspace{-2pt} & 0.294293 & 37 \\ 
\hspace{-4pt}[0.944, 0.945]\hspace{-2pt} & 0.294293 & 37 \\ 
\hspace{-4pt}[0.945, 0.946]\hspace{-2pt} & 0.297475 & 37 \\ 
\hspace{-4pt}[0.946, 0.947]\hspace{-2pt} & 0.294293 & 37 \\ 
\hspace{-4pt}[0.947, 0.948]\hspace{-2pt} & 0.291706 & 39 \\ 
\hspace{-4pt}[0.948, 0.949]\hspace{-2pt} & 0.291706 & 39 \\ 
\hspace{-4pt}[0.949, 0.950]\hspace{-2pt} & 0.288442 & 39 \\ 
\hspace{-4pt}[0.950, 0.951]\hspace{-2pt} & 0.291706 & 39 \\ 
\hspace{-4pt}[0.951, 0.952]\hspace{-2pt} & 0.291706 & 39 \\ 
\hspace{-4pt}[0.952, 0.953]\hspace{-2pt} & 0.291706 & 39 \\ 
\hspace{-4pt}[0.953, 0.954]\hspace{-2pt} & 0.294293 & 37 \\ 
\hspace{-4pt}[0.954, 0.955]\hspace{-2pt} & 0.294293 & 37 \\ 
\hspace{-4pt}[0.955, 0.956]\hspace{-2pt} & 0.291706 & 39 \\ 
\hspace{-4pt}[0.956, 0.957]\hspace{-2pt} & 0.291706 & 37 \\ 
\hspace{-4pt}[0.957, 0.958]\hspace{-2pt} & 0.285349 & 39 \\ 
\hspace{-4pt}[0.958, 0.959]\hspace{-2pt} & 0.291706 & 39 \\ 
\hspace{-4pt}[0.959, 0.960]\hspace{-2pt} & 0.291706 & 39 \\ 
\hspace{-4pt}[0.960, 0.961]\hspace{-2pt} & 0.291706 & 39 \\ 
\hspace{-4pt}[0.961, 0.962]\hspace{-2pt} & 0.291706 & 39 \\ 
\hspace{-4pt}[0.962, 0.963]\hspace{-2pt} & 0.288442 & 39 \\ 
\hspace{-4pt}[0.963, 0.964]\hspace{-2pt} & 0.288442 & 39 \\ 
\hspace{-4pt}[0.964, 0.965]\hspace{-2pt} & 0.291706 & 39 \\ 
\hspace{-4pt}[0.965, 0.966]\hspace{-2pt} & 0.291706 & 39 \\ 
\hspace{-4pt}[0.966, 0.967]\hspace{-2pt} & 0.291706 & 39 \\ 
\hspace{-4pt}[0.967, 0.968]\hspace{-2pt} & 0.288442 & 39 \\ 
\hspace{-4pt}[0.968, 0.969]\hspace{-2pt} & 0.288442 & 39 \\ 
\hspace{-4pt}[0.969, 0.970]\hspace{-2pt} & 0.285349 & 39 \\ 
\hspace{-4pt}[0.970, 0.971]\hspace{-2pt} & 0.291706 & 39 \\ 
\hspace{-4pt}[0.971, 0.972]\hspace{-2pt} & 0.285349 & 39 \\ 
\hspace{-4pt}[0.972, 0.973]\hspace{-2pt} & 0.288442 & 39 \\ 
\hspace{-4pt}[0.973, 0.974]\hspace{-2pt} & 0.288442 & 39 \\ 
\hspace{-4pt}[0.974, 0.975]\hspace{-2pt} & 0.288442 & 39 \\ 
\hspace{-4pt}[0.975, 0.976]\hspace{-2pt} & 0.288442 & 39 \\ 
\hspace{-4pt}[0.976, 0.977]\hspace{-2pt} & 0.288442 & 39 \\ 
\hspace{-4pt}[0.977, 0.978]\hspace{-2pt} & 0.288442 & 39 \\ 
\hspace{-4pt}[0.978, 0.979]\hspace{-2pt} & 0.291706 & 39 \\ 
\hspace{-4pt}[0.979, 0.980]\hspace{-2pt} & 0.285349 & 39 \\ 
\end{tabular}
\hspace{0.2 cm}
\begin{tabular}{cc|c}
\hspace{-2pt}$\varepsilon$ interval & $h(f_\varepsilon)\geq$ & \hspace{-2pt}sym\hspace{-2pt} \\[1pt] \hline \\[-6pt] 
\hspace{-4pt}[0.980, 0.981]\hspace{-2pt} & 0.285349 & 39 \\ 
\hspace{-4pt}[0.981, 0.982]\hspace{-2pt} & 0.285349 & 39 \\ 
\hspace{-4pt}[0.982, 0.983]\hspace{-2pt} & 0.285349 & 39 \\ 
\hspace{-4pt}[0.983, 0.984]\hspace{-2pt} & 0.285349 & 39 \\ 
\hspace{-4pt}[0.984, 0.985]\hspace{-2pt} & 0.285349 & 39 \\ 
\hspace{-4pt}[0.985, 0.986]\hspace{-2pt} & 0.288442 & 39 \\ 
\hspace{-4pt}[0.986, 0.987]\hspace{-2pt} & 0.285349 & 39 \\ 
\hspace{-4pt}[0.987, 0.988]\hspace{-2pt} & 0.291706 & 39 \\ 
\hspace{-4pt}[0.988, 0.989]\hspace{-2pt} & 0.285349 & 39 \\ 
\hspace{-4pt}[0.989, 0.990]\hspace{-2pt} & 0.285349 & 39 \\ 
\hspace{-4pt}[0.990, 0.991]\hspace{-2pt} & 0.285349 & 39 \\ 
\hspace{-4pt}[0.991, 0.992]\hspace{-2pt} & 0.288442 & 39 \\ 
\hspace{-4pt}[0.992, 0.993]\hspace{-2pt} & 0.285349 & 39 \\ 
\hspace{-4pt}[0.993, 0.994]\hspace{-2pt} & 0.291706 & 39 \\ 
\hspace{-4pt}[0.994, 0.995]\hspace{-2pt} & 0.291706 & 39 \\ 
\hspace{-4pt}[0.995, 0.996]\hspace{-2pt} & 0.291706 & 39 \\ 
\hspace{-4pt}[0.996, 0.997]\hspace{-2pt} & 0.288442 & 39 \\ 
\hspace{-4pt}[0.997, 0.998]\hspace{-2pt} & 0.291706 & 39 \\ 
\hspace{-4pt}[0.998, 0.999]\hspace{-2pt} & 0.288442 & 39 \\ 
\hspace{-4pt}[0.999, 1.000]\hspace{-2pt} & 0.288442 & 39 \\ 
\hspace{-4pt}[1.000, 1.005]\hspace{-2pt} & 0.313677 & 35 \\ 
\hspace{-4pt}[1.005, 1.010]\hspace{-2pt} & 0.300202 & 35 \\ 
\hspace{-4pt}[1.010, 1.015]\hspace{-2pt} & 0.303084 & 35 \\ 
\hspace{-4pt}[1.015, 1.020]\hspace{-2pt} & 0.276723 & 27 \\ 
\hspace{-4pt}[1.020, 1.025]\hspace{-2pt} & 0.300202 & 35 \\ 
\hspace{-4pt}[1.025, 1.030]\hspace{-2pt} & 0.297053 & 37 \\ 
\hspace{-4pt}[1.030, 1.035]\hspace{-2pt} & 0.297053 & 37 \\ 
\hspace{-4pt}[1.035, 1.040]\hspace{-2pt} & 0.300202 & 35 \\ 
\hspace{-4pt}[1.040, 1.045]\hspace{-2pt} & 0.291706 & 37 \\ 
\hspace{-4pt}[1.045, 1.050]\hspace{-2pt} & 0.274243 & 36 \\ 
\hspace{-4pt}[1.050, 1.055]\hspace{-2pt} & 0.297053 & 39 \\ 
\hspace{-4pt}[1.055, 1.060]\hspace{-2pt} & 0.300202 & 37 \\ 
\hspace{-4pt}[1.060, 1.065]\hspace{-2pt} & 0.291706 & 37 \\ 
\hspace{-4pt}[1.065, 1.070]\hspace{-2pt} & 0.267938 & 39 \\ 
\hspace{-4pt}[1.070, 1.075]\hspace{-2pt} & 0.297053 & 36 \\ 
\hspace{-4pt}[1.075, 1.080]\hspace{-2pt} & 0.294293 & 37 \\ 
\hspace{-4pt}[1.080, 1.085]\hspace{-2pt} & 0.300202 & 37 \\ 
\hspace{-4pt}[1.085, 1.090]\hspace{-2pt} & 0.294293 & 37 \\ 
\hspace{-4pt}[1.090, 1.095]\hspace{-2pt} & 0.291706 & 37 \\ 
\hspace{-4pt}[1.095, 1.100]\hspace{-2pt} & 0.294293 & 37 \\ 
\hspace{-4pt}[1.100, 1.105]\hspace{-2pt} & 0.294293 & 37 \\ 
\hspace{-4pt}[1.105, 1.110]\hspace{-2pt} & 0.272833 & 39 \\ 
\hspace{-4pt}[1.110, 1.115]\hspace{-2pt} & 0.297053 & 37 \\ 
\hspace{-4pt}[1.115, 1.120]\hspace{-2pt} & 0.297053 & 37 \\ 
\hspace{-4pt}[1.120, 1.125]\hspace{-2pt} & 0.297053 & 37 \\ 
\hspace{-4pt}[1.125, 1.130]\hspace{-2pt} & 0.291706 & 37 \\ 
\hspace{-4pt}[1.130, 1.135]\hspace{-2pt} & 0.297475 & 37 \\ 
\hspace{-4pt}[1.135, 1.140]\hspace{-2pt} & 0.291706 & 37 \\ 
\hspace{-4pt}[1.140, 1.145]\hspace{-2pt} & 0.291706 & 37 \\ 
\hspace{-4pt}[1.145, 1.150]\hspace{-2pt} & 0.297053 & 37 \\ 
\hspace{-4pt}[1.150, 1.155]\hspace{-2pt} & 0.294293 & 37 \\ 
\hspace{-4pt}[1.155, 1.160]\hspace{-2pt} & 0.297053 & 37 \\ 
\hspace{-4pt}[1.160, 1.165]\hspace{-2pt} & 0.297053 & 37 \\ 
\hspace{-4pt}[1.165, 1.170]\hspace{-2pt} & 0.297053 & 37 \\ 
\hspace{-4pt}[1.170, 1.175]\hspace{-2pt} & 0.297053 & 37 \\ 
\hspace{-4pt}[1.175, 1.180]\hspace{-2pt} & 0.294293 & 37 \\ 
\end{tabular}
\end{center}
}
{\scriptsize
\begin{center}
\begin{tabular}{cc|c}
\hspace{-2pt}$\varepsilon$ interval & $h(f_\varepsilon)\geq$ & \hspace{-2pt}sym\hspace{-2pt} \\[1pt] \hline \\[-6pt] 
\hspace{-4pt}[1.180, 1.185]\hspace{-2pt} & 0.297475 & 37 \\ 
\hspace{-4pt}[1.185, 1.190]\hspace{-2pt} & 0.300202 & 37 \\ 
\hspace{-4pt}[1.190, 1.195]\hspace{-2pt} & 0.300202 & 37 \\ 
\hspace{-4pt}[1.195, 1.200]\hspace{-2pt} & 0.297053 & 37 \\ 
\hspace{-4pt}[1.200, 1.205]\hspace{-2pt} & 0.353423 & 31 \\ 
\hspace{-4pt}[1.205, 1.210]\hspace{-2pt} & 0.349617 & 29 \\ 
\hspace{-4pt}[1.210, 1.215]\hspace{-2pt} & 0.344595 & 29 \\ 
\hspace{-4pt}[1.215, 1.220]\hspace{-2pt} & 0.340662 & 31 \\ 
\hspace{-4pt}[1.220, 1.225]\hspace{-2pt} & 0.344595 & 29 \\ 
\hspace{-4pt}[1.225, 1.230]\hspace{-2pt} & 0.353423 & 29 \\ 
\hspace{-4pt}[1.230, 1.235]\hspace{-2pt} & 0.340662 & 31 \\ 
\hspace{-4pt}[1.235, 1.240]\hspace{-2pt} & 0.344595 & 29 \\ 
\hspace{-4pt}[1.240, 1.245]\hspace{-2pt} & 0.339892 & 31 \\ 
\hspace{-4pt}[1.245, 1.250]\hspace{-2pt} & 0.339892 & 29 \\ 
\hspace{-4pt}[1.250, 1.255]\hspace{-2pt} & 0.344595 & 31 \\ 
\hspace{-4pt}[1.255, 1.260]\hspace{-2pt} & 0.344595 & 31 \\ 
\hspace{-4pt}[1.260, 1.265]\hspace{-2pt} & 0.339892 & 31 \\ 
\hspace{-4pt}[1.265, 1.270]\hspace{-2pt} & 0.349617 & 31 \\ 
\hspace{-4pt}[1.270, 1.275]\hspace{-2pt} & 0.349617 & 31 \\ 
\hspace{-4pt}[1.275, 1.280]\hspace{-2pt} & 0.353423 & 31 \\ 
\hspace{-4pt}[1.280, 1.285]\hspace{-2pt} & 0.344595 & 31 \\ 
\hspace{-4pt}[1.285, 1.290]\hspace{-2pt} & 0.344595 & 31 \\ 
\hspace{-4pt}[1.290, 1.295]\hspace{-2pt} & 0.339892 & 31 \\ 
\hspace{-4pt}[1.295, 1.300]\hspace{-2pt} & 0.344595 & 31 \\ 
\hspace{-4pt}[1.300, 1.305]\hspace{-2pt} & 0.339892 & 31 \\ 
\hspace{-4pt}[1.305, 1.310]\hspace{-2pt} & 0.344595 & 31 \\ 
\hspace{-4pt}[1.310, 1.315]\hspace{-2pt} & 0.344595 & 31 \\ 
\hspace{-4pt}[1.315, 1.320]\hspace{-2pt} & 0.344595 & 31 \\ 
\hspace{-4pt}[1.320, 1.325]\hspace{-2pt} & 0.344595 & 31 \\ 
\hspace{-4pt}[1.325, 1.330]\hspace{-2pt} & 0.344595 & 31 \\ 
\hspace{-4pt}[1.330, 1.335]\hspace{-2pt} & 0.339892 & 29 \\ 
\hspace{-4pt}[1.335, 1.340]\hspace{-2pt} & 0.339892 & 31 \\ 
\hspace{-4pt}[1.340, 1.345]\hspace{-2pt} & 0.339892 & 31 \\ 
\hspace{-4pt}[1.345, 1.350]\hspace{-2pt} & 0.339892 & 31 \\ 
\hspace{-4pt}[1.350, 1.355]\hspace{-2pt} & 0.339892 & 31 \\ 
\hspace{-4pt}[1.355, 1.360]\hspace{-2pt} & 0.344595 & 31 \\ 
\hspace{-4pt}[1.360, 1.365]\hspace{-2pt} & 0.339892 & 31 \\ 
\hspace{-4pt}[1.365, 1.370]\hspace{-2pt} & 0.339892 & 31 \\ 
\hspace{-4pt}[1.370, 1.375]\hspace{-2pt} & 0.339892 & 31 \\ 
\hspace{-4pt}[1.375, 1.380]\hspace{-2pt} & 0.348848 & 31 \\ 
\hspace{-4pt}[1.380, 1.385]\hspace{-2pt} & 0.344595 & 31 \\ 
\hspace{-4pt}[1.385, 1.390]\hspace{-2pt} & 0.349617 & 31 \\ 
\hspace{-4pt}[1.390, 1.395]\hspace{-2pt} & 0.349617 & 31 \\ 
\hspace{-4pt}[1.395, 1.400]\hspace{-2pt} & 0.344595 & 31 \\ 
\hspace{-4pt}[1.400, 1.405]\hspace{-2pt} & 0.344595 & 31 \\ 
\hspace{-4pt}[1.405, 1.410]\hspace{-2pt} & 0.349617 & 31 \\ 
\hspace{-4pt}[1.410, 1.415]\hspace{-2pt} & 0.349617 & 31 \\ 
\hspace{-4pt}[1.415, 1.420]\hspace{-2pt} & 0.344595 & 31 \\ 
\hspace{-4pt}[1.420, 1.425]\hspace{-2pt} & 0.339892 & 31 \\ 
\hspace{-4pt}[1.425, 1.430]\hspace{-2pt} & 0.339892 & 31 \\ 
\hspace{-4pt}[1.430, 1.435]\hspace{-2pt} & 0.344595 & 31 \\ 
\hspace{-4pt}[1.435, 1.440]\hspace{-2pt} & 0.344595 & 31 \\ 
\hspace{-4pt}[1.440, 1.445]\hspace{-2pt} & 0.339892 & 29 \\ 
\hspace{-4pt}[1.445, 1.450]\hspace{-2pt} & 0.339892 & 31 \\ 
\hspace{-4pt}[1.450, 1.455]\hspace{-2pt} & 0.344595 & 31 \\ 
\hspace{-4pt}[1.455, 1.460]\hspace{-2pt} & 0.344595 & 31 \\ 
\end{tabular}
\hspace{0.2 cm}
\begin{tabular}{cc|c}
\hspace{-2pt}$\varepsilon$ interval & $h(f_\varepsilon)\geq$ & \hspace{-2pt}sym\hspace{-2pt} \\[1pt] \hline \\[-6pt] 
\hspace{-4pt}[1.460, 1.465]\hspace{-2pt} & 0.344595 & 31 \\ 
\hspace{-4pt}[1.465, 1.470]\hspace{-2pt} & 0.344595 & 31 \\ 
\hspace{-4pt}[1.470, 1.475]\hspace{-2pt} & 0.344595 & 31 \\ 
\hspace{-4pt}[1.475, 1.480]\hspace{-2pt} & 0.344595 & 31 \\ 
\hspace{-4pt}[1.480, 1.485]\hspace{-2pt} & 0.339892 & 31 \\ 
\hspace{-4pt}[1.485, 1.490]\hspace{-2pt} & 0.339892 & 31 \\ 
\hspace{-4pt}[1.490, 1.495]\hspace{-2pt} & 0.344595 & 31 \\ 
\hspace{-4pt}[1.495, 1.500]\hspace{-2pt} & 0.344595 & 31 \\ 
\hspace{-4pt}[1.500, 1.505]\hspace{-2pt} & 0.349617 & 31 \\ 
\hspace{-4pt}[1.505, 1.510]\hspace{-2pt} & 0.349617 & 31 \\ 
\hspace{-4pt}[1.510, 1.515]\hspace{-2pt} & 0.344595 & 31 \\ 
\hspace{-4pt}[1.515, 1.520]\hspace{-2pt} & 0.344595 & 31 \\ 
\hspace{-4pt}[1.520, 1.525]\hspace{-2pt} & 0.344595 & 31 \\ 
\hspace{-4pt}[1.525, 1.530]\hspace{-2pt} & 0.362385 & 29 \\ 
\hspace{-4pt}[1.530, 1.535]\hspace{-2pt} & 0.353822 & 29 \\ 
\hspace{-4pt}[1.535, 1.540]\hspace{-2pt} & 0.353423 & 31 \\ 
\hspace{-4pt}[1.540, 1.545]\hspace{-2pt} & 0.349617 & 29 \\ 
\hspace{-4pt}[1.545, 1.550]\hspace{-2pt} & 0.349617 & 29 \\ 
\hspace{-4pt}[1.550, 1.555]\hspace{-2pt} & 0.353423 & 29 \\ 
\hspace{-4pt}[1.555, 1.560]\hspace{-2pt} & 0.349617 & 29 \\ 
\hspace{-4pt}[1.560, 1.565]\hspace{-2pt} & 0.349617 & 29 \\ 
\hspace{-4pt}[1.565, 1.570]\hspace{-2pt} & 0.349617 & 29 \\ 
\hspace{-4pt}[1.570, 1.575]\hspace{-2pt} & 0.349617 & 29 \\ 
\hspace{-4pt}[1.575, 1.580]\hspace{-2pt} & 0.344595 & 29 \\ 
\hspace{-4pt}[1.580, 1.585]\hspace{-2pt} & 0.349617 & 29 \\ 
\hspace{-4pt}[1.585, 1.590]\hspace{-2pt} & 0.349617 & 31 \\ 
\hspace{-4pt}[1.590, 1.595]\hspace{-2pt} & 0.349617 & 29 \\ 
\hspace{-4pt}[1.595, 1.600]\hspace{-2pt} & 0.349617 & 29 \\ 
\hspace{-4pt}[1.600, 1.605]\hspace{-2pt} & 0.401206 & 25 \\ 
\hspace{-4pt}[1.605, 1.610]\hspace{-2pt} & 0.394853 & 25 \\ 
\hspace{-4pt}[1.610, 1.615]\hspace{-2pt} & 0.390054 & 25 \\ 
\hspace{-4pt}[1.615, 1.620]\hspace{-2pt} & 0.394853 & 27 \\ 
\hspace{-4pt}[1.620, 1.625]\hspace{-2pt} & 0.401206 & 25 \\ 
\hspace{-4pt}[1.625, 1.630]\hspace{-2pt} & 0.394853 & 25 \\ 
\hspace{-4pt}[1.630, 1.635]\hspace{-2pt} & 0.390054 & 25 \\ 
\hspace{-4pt}[1.635, 1.640]\hspace{-2pt} & 0.406401 & 23 \\ 
\hspace{-4pt}[1.640, 1.645]\hspace{-2pt} & 0.390054 & 27 \\ 
\hspace{-4pt}[1.645, 1.650]\hspace{-2pt} & 0.390054 & 25 \\ 
\hspace{-4pt}[1.650, 1.655]\hspace{-2pt} & 0.390054 & 25 \\ 
\hspace{-4pt}[1.655, 1.660]\hspace{-2pt} & 0.390054 & 27 \\ 
\hspace{-4pt}[1.660, 1.665]\hspace{-2pt} & 0.390054 & 27 \\ 
\hspace{-4pt}[1.665, 1.670]\hspace{-2pt} & 0.394853 & 25 \\ 
\hspace{-4pt}[1.670, 1.675]\hspace{-2pt} & 0.394853 & 25 \\ 
\hspace{-4pt}[1.675, 1.680]\hspace{-2pt} & 0.394853 & 25 \\ 
\hspace{-4pt}[1.680, 1.685]\hspace{-2pt} & 0.396415 & 25 \\ 
\hspace{-4pt}[1.685, 1.690]\hspace{-2pt} & 0.390054 & 27 \\ 
\hspace{-4pt}[1.690, 1.695]\hspace{-2pt} & 0.390054 & 25 \\ 
\hspace{-4pt}[1.695, 1.700]\hspace{-2pt} & 0.390054 & 25 \\ 
\hspace{-4pt}[1.700, 1.705]\hspace{-2pt} & 0.390054 & 25 \\ 
\hspace{-4pt}[1.705, 1.710]\hspace{-2pt} & 0.394853 & 25 \\ 
\hspace{-4pt}[1.710, 1.715]\hspace{-2pt} & 0.394853 & 25 \\ 
\hspace{-4pt}[1.715, 1.720]\hspace{-2pt} & 0.401206 & 25 \\ 
\hspace{-4pt}[1.720, 1.725]\hspace{-2pt} & 0.401206 & 27 \\ 
\hspace{-4pt}[1.725, 1.730]\hspace{-2pt} & 0.390054 & 25 \\ 
\hspace{-4pt}[1.730, 1.735]\hspace{-2pt} & 0.390054 & 25 \\ 
\hspace{-4pt}[1.735, 1.740]\hspace{-2pt} & 0.390054 & 27 \\ 
\end{tabular}
\hspace{0.2 cm}
\begin{tabular}{cc|c}
\hspace{-2pt}$\varepsilon$ interval & $h(f_\varepsilon)\geq$ & \hspace{-2pt}sym\hspace{-2pt} \\[1pt] \hline \\[-6pt] 
\hspace{-4pt}[1.740, 1.745]\hspace{-2pt} & 0.394853 & 25 \\ 
\hspace{-4pt}[1.745, 1.750]\hspace{-2pt} & 0.394853 & 25 \\ 
\hspace{-4pt}[1.750, 1.755]\hspace{-2pt} & 0.390054 & 25 \\ 
\hspace{-4pt}[1.755, 1.760]\hspace{-2pt} & 0.394853 & 25 \\ 
\hspace{-4pt}[1.760, 1.765]\hspace{-2pt} & 0.390054 & 27 \\ 
\hspace{-4pt}[1.765, 1.770]\hspace{-2pt} & 0.394853 & 25 \\ 
\hspace{-4pt}[1.770, 1.775]\hspace{-2pt} & 0.406401 & 25 \\ 
\hspace{-4pt}[1.775, 1.780]\hspace{-2pt} & 0.401206 & 25 \\ 
\hspace{-4pt}[1.780, 1.785]\hspace{-2pt} & 0.394853 & 25 \\ 
\hspace{-4pt}[1.785, 1.790]\hspace{-2pt} & 0.394853 & 25 \\ 
\hspace{-4pt}[1.790, 1.795]\hspace{-2pt} & 0.401206 & 25 \\ 
\hspace{-4pt}[1.795, 1.800]\hspace{-2pt} & 0.390054 & 25 \\ 
\hspace{-4pt}[1.800, 1.805]\hspace{-2pt} & 0.401206 & 25 \\ 
\hspace{-4pt}[1.805, 1.810]\hspace{-2pt} & 0.390054 & 25 \\ 
\hspace{-4pt}[1.810, 1.815]\hspace{-2pt} & 0.390054 & 27 \\ 
\hspace{-4pt}[1.815, 1.820]\hspace{-2pt} & 0.401206 & 25 \\ 
\hspace{-4pt}[1.820, 1.825]\hspace{-2pt} & 0.390054 & 25 \\ 
\hspace{-4pt}[1.825, 1.830]\hspace{-2pt} & 0.390054 & 25 \\ 
\hspace{-4pt}[1.830, 1.835]\hspace{-2pt} & 0.401206 & 27 \\ 
\hspace{-4pt}[1.835, 1.840]\hspace{-2pt} & 0.401206 & 25 \\ 
\hspace{-4pt}[1.840, 1.845]\hspace{-2pt} & 0.401206 & 25 \\ 
\hspace{-4pt}[1.845, 1.850]\hspace{-2pt} & 0.390054 & 27 \\ 
\hspace{-4pt}[1.850, 1.855]\hspace{-2pt} & 0.401206 & 25 \\ 
\hspace{-4pt}[1.855, 1.860]\hspace{-2pt} & 0.401206 & 25 \\ 
\hspace{-4pt}[1.860, 1.865]\hspace{-2pt} & 0.401206 & 27 \\ 
\hspace{-4pt}[1.865, 1.870]\hspace{-2pt} & 0.411995 & 25 \\ 
\hspace{-4pt}[1.870, 1.875]\hspace{-2pt} & 0.406401 & 25 \\ 
\hspace{-4pt}[1.875, 1.880]\hspace{-2pt} & 0.390054 & 25 \\ 
\hspace{-4pt}[1.880, 1.885]\hspace{-2pt} & 0.394853 & 25 \\ 
\hspace{-4pt}[1.885, 1.890]\hspace{-2pt} & 0.406401 & 25 \\ 
\hspace{-4pt}[1.890, 1.895]\hspace{-2pt} & 0.406401 & 25 \\ 
\hspace{-4pt}[1.895, 1.900]\hspace{-2pt} & 0.401206 & 25 \\ 
\hspace{-4pt}[1.900, 1.905]\hspace{-2pt} & 0.406401 & 25 \\ 
\hspace{-4pt}[1.905, 1.910]\hspace{-2pt} & 0.401206 & 25 \\ 
\hspace{-4pt}[1.910, 1.915]\hspace{-2pt} & 0.401206 & 25 \\ 
\hspace{-4pt}[1.915, 1.920]\hspace{-2pt} & 0.401206 & 25 \\ 
\hspace{-4pt}[1.920, 1.925]\hspace{-2pt} & 0.401206 & 25 \\ 
\hspace{-4pt}[1.925, 1.930]\hspace{-2pt} & 0.390054 & 25 \\ 
\hspace{-4pt}[1.930, 1.935]\hspace{-2pt} & 0.401206 & 25 \\ 
\hspace{-4pt}[1.935, 1.940]\hspace{-2pt} & 0.406401 & 25 \\ 
\hspace{-4pt}[1.940, 1.945]\hspace{-2pt} & 0.390054 & 25 \\ 
\hspace{-4pt}[1.945, 1.950]\hspace{-2pt} & 0.394853 & 25 \\ 
\hspace{-4pt}[1.950, 1.955]\hspace{-2pt} & 0.411995 & 25 \\ 
\hspace{-4pt}[1.955, 1.960]\hspace{-2pt} & 0.406401 & 25 \\ 
\hspace{-4pt}[1.960, 1.965]\hspace{-2pt} & 0.406401 & 25 \\ 
\hspace{-4pt}[1.965, 1.970]\hspace{-2pt} & 0.411995 & 25 \\ 
\hspace{-4pt}[1.970, 1.975]\hspace{-2pt} & 0.406401 & 23 \\ 
\hspace{-4pt}[1.975, 1.980]\hspace{-2pt} & 0.411995 & 23 \\ 
\hspace{-4pt}[1.980, 1.985]\hspace{-2pt} & 0.411995 & 25 \\ 
\hspace{-4pt}[1.985, 1.990]\hspace{-2pt} & 0.411995 & 23 \\ 
\hspace{-4pt}[1.990, 1.995]\hspace{-2pt} & 0.411995 & 25 \\ 
\hspace{-4pt}[1.995, 2.000]\hspace{-2pt} & 0.411995 & 25 \\ 
\end{tabular}
\end{center}
}

%% file: STDmap-arxiv.bbl
\providecommand{\bysame}{\leavevmode\hbox to3em{\hrulefill}\thinspace}
\providecommand{\MR}{\relax\ifhmode\unskip\space\fi MR }
\providecommand{\MRhref}[2]{%
  \href{http://www.ams.org/mathscinet-getitem?mr=#1}{#2}
}
\providecommand{\href}[2]{#2}
\begin{thebibliography}{NBGM08}

\bibitem[BCG07]{hungarian1}
Bal{\'a}zs B{\'a}nhelyi, Tibor Csendes, and Barnabas~M. Garay,
  \emph{Optimization and the {M}iranda approach in detecting horseshoe-type
  chaos by computer}, Internat. J. Bifur. Chaos Appl. Sci. Engrg. \textbf{17}
  (2007), no.~3, 735--747. \MR{MR2324978 (2008d:37047)}

\bibitem[BCGH08]{hungarian2}
Bal{\'a}zs B{\'a}nhelyi, Tibor Csendes, Barnabas~M. Garay, and L{\'a}szl{\'o}
  Hatvani, \emph{A computer-assisted proof of {$\Sigma\sb 3$}-chaos in the
  forced damped pendulum equation}, SIAM J. Appl. Dyn. Syst. \textbf{7} (2008),
  no.~3, 843--867. \MR{MR2443025 (2009g:37029)}

\bibitem[CFdlL05]{parameterization3}
Xavier Cabr{\'e}, Ernest Fontich, and Rafael de~la Llave, \emph{The
  parameterization method for invariant manifolds. {III}. {O}verview and
  applications}, J. Differential Equations \textbf{218} (2005), no.~2,
  444--515. \MR{MR2177465 (2007b:37057)}

\bibitem[cho]{chomp}
\emph{{C}omputational {H}omology {P}roject ({CH}om{P})},
  (http://chomp.rutgers.edu).

\bibitem[DFJ01]{GAIO}
Michael Dellnitz, Gary Froyland, and Oliver Junge, \emph{The algorithms behind
  {GAIO}-set oriented numerical methods for dynamical systems}, Ergodic theory,
  analysis, and efficient simulation of dynamical systems, Springer, Berlin,
  2001, pp.~145--174, 805--807. \MR{MR1850305 (2002k:65217)}

\bibitem[DFT08]{DFT}
Sarah Day, Rafael Frongillo, and Rodrigo Trevi{\~{n}}o, \emph{Algorithms for
  rigorous entropy bounds and symbolic dynamics}, SIAM Journal on Applied
  Dynamical Systems \textbf{7} (2008), no.~4, 1477--1506.

\bibitem[DJM04]{sarah-kot}
S.~Day, O.~Junge, and K.~Mischaikow, \emph{A rigorous numerical method for the
  global analysis of infinite-dimensional discrete dynamical systems}, SIAM J.
  Appl. Dyn. Syst. \textbf{3} (2004), no.~2, 117--160 (electronic).
  \MR{MR2067140 (2005c:37165)}

\bibitem[DJM05]{DJM05}
Sarah Day, Oliver Junge, and Konstantin Mischaikow, \emph{Towards automated
  chaos verification}, E{QUADIFF} 2003, World Sci. Publ., Hackensack, NJ, 2005,
  pp.~157--162. \MR{MR2185012}

\bibitem[FR00]{FranksRicheson}
John Franks and David Richeson, \emph{Shift equivalence and the {C}onley
  index}, Trans. Amer. Math. Soc. \textbf{352} (2000), no.~7, 3305--3322.
  \MR{MR1665329 (2000j:37013)}

\bibitem[Fro10]{raf-plats}
R.M. Frongillo, \emph{{Topological entropy bounds for hyperbolic plateaus of
  the H$\backslash$'enon map}}, Arxiv preprint arXiv:1001.4211 (2010).

\bibitem[Gel99]{gelfreich}
V.~G. Gelfreich, \emph{A proof of the exponentially small transversality of the
  separatrices for the standard map}, Comm. Math. Phys. \textbf{201} (1999),
  no.~1, 155--216. \MR{MR1669417 (2000c:37087)}

\bibitem[Gol01]{gole}
Christophe Gol{\'e}, \emph{Symplectic twist maps}, Advanced Series in Nonlinear
  Dynamics, vol.~18, World Scientific Publishing Co. Inc., River Edge, NJ,
  2001, Global variational techniques. \MR{MR1992005 (2004f:37070)}

\bibitem[{Gre}79]{Greene}
J.~M. {Greene}, \emph{{A method for determining a stochastic transition}},
  Journal of Mathematical Physics \textbf{20} (1979), 1183--1201.

\bibitem[Jun91]{Jungreis}
Irwin Jungreis, \emph{A method for proving that monotone twist maps have no
  invariant circles}, Ergodic Theory Dynam. Systems \textbf{11} (1991), no.~1,
  79--84. \MR{MR1101085 (92c:58125)}

\bibitem[KMM04]{CompHom}
Tomasz Kaczynski, Konstantin Mischaikow, and Marian Mrozek, \emph{Computational
  homology}, Applied Mathematical Sciences, vol. 157, Springer-Verlag, New
  York, 2004. \MR{MR2028588 (2005g:55001)}

\bibitem[Kni96]{knill}
Oliver Knill, \emph{Topological entropy of standard type monotone twist maps},
  Trans. Amer. Math. Soc. \textbf{348} (1996), no.~8, 2999--3013. \MR{1373642
  (96m:58138)}

\bibitem[MJ09]{jay1}
J.D. Mireless~James, \emph{{Adaptive Set-Oriented Computation of Topological
  Horseshoe Factors in Area and Volume Preserving Maps}}, To appear in SIAM
  Journal on Applied Dynamical Systems (2009).

\bibitem[MM02]{MM:cit}
Konstantin Mischaikow and Marian Mrozek, \emph{Conley index}, Handbook of
  dynamical systems, {V}ol. 2, North-Holland, Amsterdam, 2002, pp.~393--460.
  \MR{MR1901060 (2003g:37022)}

\bibitem[NBGM08]{newhouse}
S.~Newhouse, M.~Berz, J.~Grote, and K.~Makino, \emph{On the estimation of
  topological entropy on surfaces}, Geometric and probabilistic structures in
  dynamics, Contemp. Math., vol. 469, Amer. Math. Soc., Providence, RI, 2008,
  pp.~243--270. \MR{2478474 (2010d:37026)}

\bibitem[RS88]{salamon}
Joel~W. Robbin and Dietmar Salamon, \emph{Dynamical systems, shape theory and
  the {C}onley index}, Ergodic Theory Dynam. Systems \textbf{8$\sp *$} (1988),
  no.~Charles Conley Memorial Issue, 375--393. \MR{MR967645 (89h:58094)}

\bibitem[Rum99]{intlab}
{S.M.} Rump, \emph{{INTLAB - INTerval LABoratory}}, {Developments in Reliable
  Computing} (Tibor Csendes, ed.), Kluwer Academic Publishers, Dordrecht, 1999,
  pp.~77--104.

\bibitem[Wil70]{BobWilliams}
R.~F. Williams, \emph{Classification of one dimensional attractors}, Global
  {A}nalysis ({P}roc. {S}ympos. {P}ure {M}ath., {V}ol. {XIV}, {B}erkeley,
  {C}alif., 1968), Amer. Math. Soc., Providence, R.I., 1970, pp.~341--361.
  \MR{MR0266227 (42 \#1134)}

\bibitem[YT09]{tanikawa}
Yoshihiro Yamaguchi and Kiyotaka Tanikawa, \emph{Topological entropy in a
  parameter range of the standard map}, Progress of Theoretical Physics
  \textbf{121} (2009), no.~4, 657--669.

\bibitem[ZG04]{covering1}
Piotr Zgliczy{\'n}ski and Marian Gidea, \emph{Covering relations for
  multidimensional dynamical systems}, J. Differential Equations \textbf{202}
  (2004), no.~1, 32--58. \MR{2060531 (2005c:37019a)}

\end{thebibliography}
